\documentclass{amsart}
\usepackage{a4wide,amsmath,amsthm,amssymb,amsfonts}
\usepackage{mathdots}
\usepackage{enumitem}
\usepackage{latexsym}
\usepackage[all]{xy}
\usepackage{CJK}
\usepackage{bm}
\usepackage{graphicx}
\usepackage{subfigure}
\usepackage{mathscinet}
\usepackage{mathrsfs}
\usepackage{float}
\usepackage{makecell}
\usepackage{multirow}
\usepackage{xcolor}
\usepackage{cite}
\usepackage[pagebackref]{hyperref}
\usepackage{hyperref}
\usepackage{verbatim}
\usepackage{amsrefs}

\numberwithin{equation}{section}
\newtheorem{theorem}{\bf{Theorem}}[section]
\newtheorem{proposition}[theorem]{\bf{Proposition}}
\newtheorem{definition}[theorem]{\bf{Definition}}
\newtheorem{corollary}[theorem]{\bf{Corollary}}
\newtheorem{lemma}[theorem]{\bf{Lemma}}
\newtheorem{remark}[theorem]{\bf{Remark}}

\newtheorem{conjecture}[theorem]{\bf{Conjecture}}

\newcommand{\abs}[1]{\left|#1\right|_{F}}
\newcommand{\car}[1]{\left|#1\right|}
\newcommand{\rest}{\lvert}

\newcommand{\mrgl}{\operatorname{GL}}
\newcommand{\mrsl}{\operatorname{SL}}

\newcommand{\mrm}{\operatorname{M}}
\newcommand{\mrn}{\operatorname{N}}
\newcommand{\mrTr}{\operatorname{Tr}}
\newcommand{\mrtr}{\operatorname{tr}}
\newcommand{\mrdet}{\operatorname{det}}
\newcommand{\mrdim}{\operatorname{dim}}
\newcommand{\mrhom}{\operatorname{Hom}}

\newcommand{\mrdiag}{\operatorname{diag}}
\newcommand{\mrgcd}{\operatorname{gcd}}

\newcommand{\mrirr}{\operatorname{Irr}}
\newcommand{\mrrep}{\operatorname{Rep}}

\newcommand{\mrcusp}{\operatorname{Cusp}}
\newcommand{\mrsqrt}{\operatorname{Sqrt}}
\newcommand{\mrtemp}{\operatorname{Temp}}

\newcommand{\mrInd}{\operatorname{Ind}}
\newcommand{\mrind}{\operatorname{ind}}

\newcommand{\mrmc}{\operatorname{MC}}
\newcommand{\mrstab}{\operatorname{Stab}}

\newcommand{\mrreg}{\operatorname{reg}}
\newcommand{\mrwh}{\operatorname{Wh}}
\newcommand{\mrexp}{\operatorname{exp}}
\newcommand{\mrnil}{\operatorname{Nil}}

\newcommand{\mbz}{\mathbb{Z}}
\newcommand{\mbc}{\mathbb{C}}
\newcommand{\mbr}{\mathbb{R}}
\newcommand{\mbq}{\mathbb{Q}}
\newcommand{\mbf}{\mathbb{F}}

\newcommand{\mfp}{\mathfrak{p}}

\newcommand{\mfo}{\mathfrak{o}}
\newcommand{\mfa}{\mathfrak{a}}
\newcommand{\mfb}{\mathfrak{b}}

\newcommand{\mfg}{\mathfrak{g}}

\newcommand{\mco}{\mathcal{O}}
\newcommand{\mcc}{\mathcal{C}}

\newcommand{\dF}[1]{d_{F}(#1)}

\newcommand{\bs}[1]{\boldsymbol{#1}}

\newcommand{\wt}[1]{\widetilde{#1}}

\begin{document}
	
	\title{Local metaplectic correspondence and applications}
	\author{Jiandi Zou}
	\address{Mathematics Department, Technion -- Israel Institute of Technology, Haifa 3200003, Israel}
	\email{idealzjd@gmail.com}
	
	\begin{abstract}
		
		We revisit the local metaplectic correspondence previously constructed and studied by Flicker and Kazhdan. After restoring and generalizing some of their results, we get several interesting applications to the representation theory of a Kazhdan-Patterson covering group over a $p$-adic field, including a full criterion of the irreducibility of the Bernstein-Zelevinsky product of two cuspidal representations, the classification of essentially square integrable and tempered representations, and more interestingly, the calculation of the Whittaker dimension of an irreducible representation based on a conjecture.
		
	\end{abstract}	
	\maketitle
	\setcounter{tocdepth}{1}
	\tableofcontents
	
	\section{Introduction}
	
	Motivated by the famous program of Robert Langlands, a main topic in the representation theory 
	is the study of the category of smooth representations of a $p$-adic reductive group, and functorial maps between such categories for different $p$-adic reductive groups over a fixed base field. We leave \cite{borel1979automorphic} for an expository introduction. Still motivated by the hope of generalizing Langlands' philosophy to larger classes of $p$-adic groups, the representation theory of finite central covering groups of $p$-adic reductive groups gradually became shining in the past decades. Correspondingly, we leave \cite{gan2018groups} for a historical introduction.  
	
	The main focus of this article is more concrete: we study a certain functorial map, called the local metaplectic correspondence, from the set of (certain) representations of a Kazhdan-Patterson covering group of a $p$-adic general linear group, to that of a $p$-adic general linear group. Our goal is, on the one hand, to restore and slightly generalize the existing theory, and on the other hand, to dig out interesting applications of this fascinating map.
	
	We give a brief introduction of the above theory. Let $F$ be a $p$-adic field and assume that the subgroup of $F^{\times}$ of $n$-th roots of unity, denoted by $\mu_{n}$, is of order $n$. We consider a certain $n$-fold covering group $\wt{G_{r}}$ as a central extension of $G_{r}:=\mrgl_{r}(F)$ by $\mu_{n}$, constructed by Kazhdan-Patterson in their seminal work \cite{kazhdan1984metaplectic} (see \S \ref{subsectionKPcover} for more details). In particular when $r=1$ we get a central extension  of $F^{\times}$ by $\mu_{n}$ twisted by a power of the $n$-th Hilbert symbol, and when $r=2$ we get the cover considered by Kubota \cite{kubota1967topological} (which is for $\mrsl_{2}(F)$ but can be extended to $\mrgl_{2}(F)$). 
	
	The (local) metaplectic correspondence (for essentially square integrable representations) is a map from 
	the set of genuine essentially square integrable representations of $\wt{G_{r}}$ to the set of essentially square integrable representations of $G_{r}$ that are trivial on $\mu_{n}I_{r}\subset G_r$. Such a map is determined by a formula of Harish-Chandra characters of representations and constructed using a trace formula comparison. Historically, in the case $r=n=2$ it first occurred as a map between two different classes of modular forms and is called the ``Shimura correspondence" \cite{shimura1973modular}. Using a trace formula comparison, Flicker \cite{flicker1980automorphic} constructed the metaplectic correspondence for $r=2$ and general $n$ and Flicker-Kazhdan \cite{flicker1986metaplectic} generalized the above theory to general $r$ and $n$. 
	
	Our starting point is the metaplectic correspondence for essentially square integrable representations of $\wt{G_r}$ and $G_r$ constructed in \cite{flicker1986metaplectic}. The first result of this article is to construct the metaplectic correspondence for essentially square integrable representations of Levi subgroups of $\wt{G_r}$ and $G_r$ as well (\emph{cf.} \S \ref{subsectionMCdefinition}). After that, we study its compatibility with parabolic induction and Jacquet module (\emph{cf.} \S \ref{subsectioncomppindjac}).
	
	
	In the rest of Section 3, we focus on following applications of metaplectic correspondence, including:
	\begin{itemize}
		
		\item (\emph{cf.} Proposition \ref{propvaluesrho}) For a cuspidal representation $\wt{\rho}$ of $\wt{G_r}$, we determine the exact value of the positive real number $s(\wt{\rho})$ (which is unique) such that $(\wt{\rho}\wt{\times}\wt{\rho}\nu^{s(\wt{\rho})})_{\wt{\omega}}$ is reducible. Here $\wt{\times}$ denotes the Bernstein-Zelevinsky product, $\nu(\cdot):=\abs{\mrdet(\cdot)}$ is an unramified character of $G_{r}$ and $\wt{\omega}$ is a compatible genuine character of $Z(\wt{G_{2r}})$. 
		
		\item (\emph{cf.} Proposition \ref{propmcsqrtclassification}) We classify all the essentially square integrable representations of $\wt{G_{r}}$ as well as their image under the metaplectic correspondence, in the sense of the Zelevinsky classification (\emph{cf.} \cite{kaplan2022classification}).
		
		\item (\emph{cf.} Proposition \ref{proptempclasify}, Proposition \ref{propmclifttemp}) We classify all the essentially tempered representations of $\wt{G_{r}}$ via the Zelevinsky classification, and we study a possible metaplectic correspondence for essentially tempered representations.
		
	\end{itemize}
	
	The above applications serve well as ``appetizers", and in Section 4 our ``main dish" is to calculate the dimension of the Whittaker space of an irreducible representation of $\wt{G_{r}}$. Unlike the linear case, in general for a covering group such a Whittaker space does not satisfy the multiplicity one theorem. This causes additional difficulties in studying automorphic representations and local-global compatiblities, since usually the construction of a global $L$-function, in particular its factorization as an Euler product, is based on the uniqueness of Whittaker model (\emph{cf.} \cite{jacquet1981euler} for the Rankin-Selberg integral for instance). 
	
	To achieve a reasonable theory of global automorphic representations, it is important to understand the dimension of a Whittaker space. This is illustrated in the work of Kaplan (\emph{cf.} \cite{kaplan2019doubling}*{Appendix A}, \cite{kaplan2022rankin}), using generalized doubling method to study the L-function of a pair of representations of a covering group of $G_{r}$ \footnote{We should warn the readers that in \emph{loc. cit.} a different covering group of $G_{r}$ is considered. But hopefully a similar theory could be established for a Kazhdan-Patterson covering group.}. A key idea in his argument is that, one need to pick a certain generalized Speh representation whose Whittaker dimension is one (\emph{cf.} \cite{kaplan2022rankin}*{Conjecture 14}). The consideration of such a generalized Speh representation dates back to the work of Suzuki \cite{suzuki1998distinguished} of distinguished representations. Then the Euler product factorization of the related generalized doubling integral follows from this multiplicity one fact. 
	
	We explain our idea of calculating the Whittaker dimension of an irreducible representation for a Kazhdan-Patterson covering group. In our case, the Whittaker functor is exact and compatible with the Bernstein-Zelevinsky product (\emph{cf.} Proposition \ref{propwhitdimBZprod}). Then using the Zelevinsky classification essentially we only need to calculate the Whittaker dimension of an essentially square integrable representation $\wt{\pi}$ of $\wt{G_{r}}$ (\emph{cf.} Remark \ref{remcalirredwhdim}). The Whittaker dimension of $\wt{\pi}$ can be re-explained by the Harish-Chandra germ function $c_{\wt{\pi}}$ of $\wt{\pi}$ evaluated at the identity, which is a well-known result in the linear case \cite{rodier1975modele}, \cite{moeglin1987modeles} and also generalized to the covering group case \cite{patel2015theorem}. Let $\pi$ be the metaplectic lift of $\wt{\pi}$, then using the metaplectic correspondence the above question is reduced to calculating the Harish-Chandra germ function $c_{\pi}$ of $\pi$ evaluated at diagonal elements of order dividing $n$ in $G_{r}$ (\emph{cf.} Corollary \ref{eqwhitdim}). 
	
	Let $\pi'=Z(\rho,[a,b])$ be the unique irreducible subrepresentation of the Bernstein-Zelevinsky product $\rho\nu^{a}\times\rho\nu^{a+1}\times\dots\times\rho\nu^{b}$, where $\rho$ is a cuspidal representation of $G_{r_{0}}$ and $a\leq b$ are two integers. We propose a general conjecture (\emph{cf.} Conjecture \ref{conjmain}) about the value of $c_{\pi'}$ at a diagonal element of order dividing $n$. Our conjecture has its own interest, which dates back to the classical but difficult question of calculating special values of a Harish-Chandra germ function (\emph{cf.} \cite{harishchandra1999admissible}*{Introduction}). Using this conjecture and Tadic's determinantal formula, we are able to calculate the value of $c_{\pi}$ at a diagonal element of order $n$, and thus the Whittaker dimension of $\wt{\pi}$ (\emph{cf.} Theorem \ref{thmwhitcal}). Our dimensional formula seems to be neat enough.
	
	We verify Conjecture \ref{conjmain} for the following two cases (\emph{cf.} \S \ref{subsectioneviconj}):
	\begin{itemize}
		\item $r_{0}=1$	and $\rho$ is a character of $F^{\times}$;
		\item $a=b$ and $\pi'$ is cuspidal.
	\end{itemize}
	The first case is rather direct, since $\pi'$ itself is a character. The proof of the second case is based on a rudiment usage of the simple type theory of Bushnell-Kutzko \cite{bushnell129admissible}. Hopefully our argument for the second case can be used to other groups and their supercuspidal representations. 
	
	Finally, in Section 5 we explain how our results could be used to study an analogue of \cite{kaplan2022rankin}*{Conjecture 14} for a Kazhdan-Patterson covering group.
	
	We end the introduction by comparing our result with other work of calculating the Whittaker dimension of an irreducible representation of a covering group. For a Kazhdan-Patterson covering group, to the knowledge of the author, the only known cases\footnote{Here, by ``known" we mean that a concrete value could theoretically be obtained, although sometimes there exist combinatorial difficulties. Also see Remark \ref{remcalirred}.} are principal series and their subquotients \cite{kazhdan1984metaplectic}, \cite{kazhdan1986towards}, \cite{suzuki1998distinguished}, as well as depth 0 cuspidal representations \cite{blondel1992uniqueness}. For general covering groups, still the known cases are confined to principal series and their subquotients, and depth 0 cuspidal representations. We refer to the work of F. Gao and his collaborators for more details, \emph{cf.} \cite{gao2017distinguished}, 
	\cite{gao2021Rgroup}, \cite{gao2022Rgroup}, \cite{gao2022genuine}, \cite{gao2019whittaker}.
	
	The author would like to thank Ning Guo, Max Gurevich, Eyal Kaplan, Erez Lapid, Caihua Luo, Vincent S\'echerre and Chuijia Wang for helpful correspondences or discussions, and the anonymous referees for their pertinent advice.
	This research was supported by the Israel Science Foundation (grant No. 737/20).
	
	\section{Preliminaries}
	
	\subsection{Notation}\label{subsectionnotation}
	
	Throughout this article $F$ will be a non-archimedean locally compact field of residue characteristic $p$, and beginning from Section 3 we will also assume $F$ to be of characteristic $0$. Let $\mfo_{F}$ be its ring of integers, let $\mfp_{F}$ be the maximal ideal of $\mfo_{F}$ and let $\mbf_{q}$ be its residue field. Let $\abs{\cdot}$ be the normalized absolute value of $F$. 
	
	We consider $\ell$-groups in the sense of Bernstein-Zelevinsky \cite{bernstein1976representations}. The $\ell$-groups we are considering in this paper will be the $F$-points of reductive groups as well as their finite central extensions. For an $\ell$-group $G$ and its closed subgroup $H$, we denote by $Z(G)$ the center of $G$ and by $Z_{G}(H)$ the centralizer of $H$ in $G$. In particular, we let $G_{r}=\mrgl_{r}(F)$ for a positive integer $r$. 
	
	By representations of an $\ell$-group $G$, we always mean complex smooth representations. Similarly by characters we mean one-dimensional representations. We denote by $\mrrep(G)$, $\mrirr(G)$, $\mrtemp(G)$, $\mrsqrt(G)$, $\mrcusp(G)$ the sets of equivalence classes of finite length representations, irreducible representations, irreducible essentially tempered representations, irreducible essentially square integrable representations and irreducible cuspidal representations respectively. For a closed subgroup $Z$ of $Z(G)$ and a character $\omega$ of $Z$, we denote by $\mrirr_{\omega}(G)$, $\mrtemp_{\omega}(G)$, $\mrsqrt_{\omega}(G)$, $\mrcusp_{\omega}(G)$ the corresponding subsets
	of those equivalence classes of representations with central character restricted to $Z$ being $\omega$. For a representation $\pi$ of $H$ and a positive integer $m$, we let $m\cdot\pi$ be the direct sum of $m$-copies of $\pi$. If $\pi$ is irreducible, then we let $\omega_{\pi}$ be its central character. 
	
	If $G$ is a finite central extension of a reductive group over $F$ (or even an open finite index subgroup of such a group), we let $\theta_{\pi}$ be the Harish-Chandra character of $\pi$ as a locally integrable function on $G$, that is smooth on the semi-simple regular locus of $G$. This notation can also be generalized to any finite length representations of $G$ by considering the sum of the Harish-Chandra characters of each irreducible subquotient. We also remark that if $\pi\in\mrirr(G)$, then $\pi$ is determined by $\theta_{\pi}$. Moreover if $\pi\in\mrsqrt(G)$, then $\pi$ is determined by the restriction of $\theta_{\pi}$ to the elliptic locus (\emph{cf.} for instance \cite{clozel1991invariant}).
	
	For a non-negative integer $m$ and an integer $k$, we denote by $\binom{m}{k}$ the binomial coefficient of $x^{k}$ in $(1+x)^{m}$, and by $m!$ the factorial of $m$.
	
	\subsection{Kazhdan-Patterson covering group}\label{subsectionKPcover}
	
	Fix a positive integer $n$. Let $\mu_{n}:=\mu_{n}(F)$ be the group of $n$-th roots of unity in $F$. From now on we further assume that the cardinality of $\mu_{n}$ is $n$. We let $F^{\times n}:=\{x^{n}\mid x\in F^{\times}\}$ be an open finite index subgroup of $F^{\times}$.
	We denote by 
	$$(\cdot,\cdot)_{n}:F^{\times}\times F^{\times}\rightarrow \mu_{n}$$ 
	the $n$-th order Hilbert symbol. It is a bimultiplicative antisymmetric pairing that
	defines a non-degenerate bimultiplicative pairing on $F^{\times}/F^{\times n}\times F^{\times}/F^{\times n}$ (\emph{cf.} \cite{weil1974basic}*{XIII.\S5}).
	
	For a (connected) reductive group $G$ over $F$ (we use $G$ to denote its $F$-points by abuse of notation), we consider its central extension of $\ell$-groups
	$$\xymatrix{0 \ar[r] & \mu_{n} \ar[r]^-{} & \wt{G} \ar[r]^-{\bs{p}} & G \ar[r] & 0},$$
	which is called an \emph{$n$-fold covering group}. The set of equivalence classes of such central extensions is in bijection with the continuous cohomology group $H^{2}(G,\mu_{n})$. More precisely, we fix a continuous section $\bs{s}:G\rightarrow \wt{G}$, that is, $\bs{p}\circ\bs{s}=\mathrm{id}$. Then the central extension $\wt{G}$ corresponds to the cohomology class of the 2-cocycle $\sigma:G\times G\rightarrow\mu_{n}$ satisfying
	$$\bs{s}(g_{1})\bs{s}(g_{2})=\bs{s}(g_{1}g_{2})\sigma(g_{1},g_{2}),\quad g_{1},g_{2}\in G.$$
	
	In particular, we consider the case $G=G_{r}$. Let $\sigma_{\mathrm{SL}_{r+1}}$ be the special 2-cocycle of $\mathrm{SL}_{r+1}(F)$ considered by Matsumoto \cite{matsumoto1969sous} with respect to the Steinberg symbol $(\cdot,\cdot)^{-1}_{n}$. Let $\sigma^{(0)}$ be the pull-back of $\sigma_{\mathrm{SL}_{r+1}}$ via $$G_{r}\rightarrow\mathrm{SL}_{r+1}(F),\ g\mapsto\mathrm{diag}(\mathrm{det}(g)^{-1},g).$$ 
	The explicit construction and calculation of $\sigma^{(0)}$ has been done in \cite{banks1999block}. Finally for any $c\in\mbz/n\mbz$ let $\sigma^{(c)}$ be the 2-cocycle satisfying
	$$\sigma^{(c)}(g_{1},g_{2})=\sigma^{(0)}(g_{1},g_{2})\cdot(\mathrm{det}(g_{1}),\mathrm{det}(g_{2}))_{n}^{c},\quad g_{1},g_{2}\in G_{r}.$$
	From now on we fix $c\in\mbz$, and we call the $n$-fold cover $\wt{G_{r}}$ corresponding to $\sigma^{(c)}$ a \emph{Kazhdan-Patterson covering group} of $G_{r}$. In particular $\wt{F^{\times}}$ will denote the $n$-fold cover of $F^{\times}$ with respect to the 2-cocycle $(\cdot,\cdot)_{n}^{c}$. 
	
	
	For a closed subgroup $H$ of $G_{r}$, let $\bs{p}^{-1}(H)$ or $\wt{H}$ be the preimage of $H$ in $\wt{G}$ and let $H^{(n)}=\{h\in H\mid \mrdet(h)\in F^{\times n}\}$ be an open normal subgroup of $H$ of finite index.
	
	Let $\beta$ be a composition of $r$. It means that $\beta=(r_{1},\dots,r_{k})$ for certain $k$ and positive integers $r_{i}$ such that $r=r_{1}+\dots+ r_{k}$. Let $G_{\beta}$ be the standard Levi subgroup of $G_{r}$ with respect to $\beta$. So $G_{\beta}\simeq G_{r_{1}}\times\dots\times G_{r_{k}}$ with each $G_{r_{i}}$ being regarded as a subgroup of $G_{r}$ via the embedding $g_{i}\mapsto\mrdiag(I_{r_{1}},\dots,I_{r_{i-1}},g_{i},I_{r_{i+1}},\dots, I_{r_{k}})$. 
	
	We recall the following properties concerning a Kazhdan-Patterson covering group.
	\begin{enumerate}
		\item When $r=1$, we have $\sigma^{(c)}(x,y)=(x,y)^{c}_{n}$ for $x,y\in F^{\times}$.
		
		\item (\cite{banks1999block}*{Theorem 11}) The restriction of $\sigma^{(c)}$ to $G_{\beta}\times G_{\beta}$ is given by
		\begin{align}
			&\sigma^{(c)}(\mrdiag(x_{1},\dots,x_{k}),\mrdiag(y_{1},\dots, y_{k}))= \nonumber \\
			&\prod_{i=1}^{k}\sigma^{(c)}(x_{i},y_{i})\cdot\prod_{1\leq i<j\leq k}(\mrdet(x_{i}),\mrdet(y_{j}))_{n}^{c+1}(\mrdet(x_{j}),\mrdet(y_{i}))_{n}^{c},\quad x_{i},y_{i}\in G_{r_{i}} \label{eqBLS}
		\end{align}
		\item (\cite{chinta2013metaplectic}*{\S 2.1, Lemma 1}) The center of $\wt{G_{r}}$ is given by
		\begin{equation}\label{eqZwtGr}
			Z(\wt{G_{r}})=\bs{p}^{-1}(\{\lambda I_{r}\mid \lambda^{2rc+r-1}\in F^{\times n}\}).
		\end{equation}
		
		\item By a direct calculation, \begin{equation}\label{eqZwtGrintersection}
			G_{\beta}^{(n)}\cap \bs{p}(Z(\wt{G_{r}}))=G_{r}^{(n)}\cap \bs{p}(Z(\wt{G_{r}}))=F^{\times n} I_{r}.
		\end{equation}
		
		
	\end{enumerate}
	
	We fix a faithful character $\epsilon:\mu_{n}\rightarrow \mbc^{\times}$. For a closed subgroup $H$ of $G_{r}$, a representation $\wt{\pi}$ of $\wt{H}$ is called $\epsilon$-\emph{genuine} (or simply genuine if $\epsilon$ is clear) if the restriction of $\wt{\pi}$ to $\mu_{n}$ acts by $\epsilon$. Notice that if a representation of $\wt{H}$ is not genuine, then it factors through an $n'$-fold covering group with $n'$ dividing $n$. So essentially we only need to stick to $\epsilon$-genuine representations. We also remark that the $H$-conjugation on $\wt{H}$ and $\wt{\pi}$, as well as the twist $\wt{\pi}\cdot\chi$ for $\chi$ being a character of $H$ are well-defined in the usual sense.
	
	\subsection{Metaplectic tensor product}\label{subsectionMTP}
	
	In this subsection, we rephrase the construction of metaplectic tensor product  for irreducible representations in \cite{kaplan2022classification}*{\S 4}. Our presentation here is rather down-to-earth, which in particular generalizes the previous works of Kable \cite{kable2001tensor}, Mezo \cite{mezo2004metaplectic} and Takeda \cite{takeda2016metaplectic, takeda2017remarks}.

	We keep the notations of the last subsection. Let $H_{i}$ be a closed subgroup of $G_{r_{i}}$ for each $i=1,\dots,k$ and let $H=H_{1}\times\dots\times H_{k}$ be a closed subgroup of $G_{\beta}$.  Such an $H$ is called \emph{block compatible} if the groups $\wt{H_{i}}$ are pairwise commutative as subgroups in $\wt{H}$ for $i=1,\dots,k$. In this case, the group $\wt{H}$ is isomorphic to $\wt{H_{1}}\times\dots\times\wt{H_{k}}/\Xi$ with
	$$\Xi=\{(\zeta_{1},\dots,\zeta_{k})\in\mu_{n}\times\dots\times\mu_{n}\mid \zeta_{1}\dots\zeta_{k}=1\}.$$
	If $H$ is block compatible, then for each $i$ and a genuine representation $\wt{\pi}_{i}$ of $\wt{H_{i}}$, the tensor product
	$$\wt{\pi}_{1}\otimes\dots\otimes\wt{\pi}_{k}$$
	defines a representation of $\wt{H_{1}}\times\dots\times\wt{H_{k}}$ that descends to a representation of $\wt{H}$, which we denote by $$\wt{\pi}_{1}\wt{\otimes}\dots\wt{\otimes}\wt{\pi}_{k}.$$ 
	We remark that in general $H$ is usually not block compatible, but the group $H_{1}^{(n)}\times\dots\times H_{k}^{(n)}$ is block compatible by \eqref{eqBLS}. 
	
	Now for each $i=1,\dots,k$ let $\gamma_{i}$ be a composition of $r_{i}$ and let $\gamma=\gamma_{1}+\dots+\gamma_{k}$ be a composition of $r$ as a refinement of $\beta$. Given $\wt{\pi}_{i}\in\mrirr_{\epsilon}(\wt{G_{\gamma_{i}}})$ for each $i$, our goal is to construct a representation in $\mrirr_{\epsilon}(\wt{G_{\gamma}})$ as some kind of ``tensor product" of $\wt{\pi}_{i}$. This cannot be achieved directly, since $G_{\gamma}=G_{\gamma_{1}}\times\dots\times G_{\gamma_{k}}$ is not block compatible. But instead, its open and finite index subgroup 
	$$G_{\gamma,\beta}^{(n)}:=G_{\gamma_{1}}^{(n)}\times\dots\times G_{\gamma_{k}}^{(n)} $$
	is block compatible. In particular this notation is valid for $\gamma=\beta$. So we may first restrict each $\wt{\pi}_{i}$ to $\wt{G_{\gamma_{i}}^{(n)}}$, which we denote by $\wt{\pi}_{i}^{(n)}$. Then we take the tensor product
	$\wt{\pi}_{1}^{(n)}\wt{\otimes}\dots\wt{\otimes}\wt{\pi}_{k}^{(n)}$ as a representation of $\wt{G_{\gamma,\beta}^{(n)}}$. Finally we consider the induction $\mrInd_{\wt{G_{\gamma,\beta}^{(n)}}}^{\wt{G_{\gamma}}}(\wt{\pi}_{1}^{(n)}\wt{\otimes}\dots\wt{\otimes}\wt{\pi}_{k}^{(n)})$ and take one of its irreducible component as the ``tensor product" we want. Bootstrapping the above procedure gives us as the following theorem. 
	
	\begin{theorem}[\cite{kaplan2022classification}*{Proposition 3.10 and \S 4.2}]\label{thmMTP}
		
		Let $\wt{\pi}\in\mrirr_{\epsilon}(\wt{G_{\gamma}})$ and let $\wt{\pi}_{i}\in\mrirr_{\epsilon}(\wt{G_{\gamma_{i}}})$ for each $i=1,\dots,k$.
		
		\begin{enumerate}
			\item  The center of $\wt{Z(G_{\beta})}$ is $Z(\wt{G_{r}})\wt{Z(G_{\beta,\beta}^{(n)})}$, and $[\wt{Z(G_{\beta})}:Z(\wt{G_{r}})\wt{Z(G_{\beta,\beta}^{(n)})}]$ is the square of a positive integer $n_{\beta}$.
			
			\item Let $\wt{\tau}$ be any irreducible component of $\wt{\pi}^{(n)}:=\wt{\pi}\rest_{\wt{G_{\gamma,\beta}^{(n)}}}$, then we have
			\begin{equation}\label{eqpirestGgamman}
				\wt{\pi}^{(n)}\simeq n_{\beta}\cdot\bigoplus_{g\in\wt{G_{\gamma}}/\wt{Z(G_{\beta})}\wt{G_{\gamma,\beta}^{(n)}}}\wt{\tau}^{g},
			\end{equation}
			where those $\wt{\tau}^{g}$ on the right-hand side are pairwise inequivalent.
			We let $m_{\beta}:=n_{\beta}\cdot[\wt{G_{\gamma}}:\wt{Z(G_{\beta})}\wt{G_{\gamma,\beta}^{(n)}}]$ which is the total multiplicity of $\wt{\pi}^{(n)}$.
			
			\item We call a character $\wt{\omega}$ of $Z(\wt{G_{r}})$ \emph{compatible} (with respect to $(\wt{\pi}_{1},\dots,\wt{\pi}_{k})$) if the following compatibility condition is satisfied
			\begin{equation}\label{eqomegacompatible}
				\prod_{i=1}^{k}\omega_{\wt{\pi}_{i}}(\bs{s}(\lambda I_{r_{i}}))=\wt{\omega}(\bs{s}(\lambda I_{r}))\quad\text{for any}\ \lambda\in F^{\times n}.
			\end{equation}
			Then we have
			\begin{equation}\label{eqMTPdef}
				\mrInd_{\wt{G_{\gamma,\beta}^{(n)}}}^{\wt{G_{\gamma}}}(\wt{\pi}_{1}^{(n)}\wt{\otimes}\dots\wt{\otimes}\wt{\pi}_{k}^{(n)})\simeq M_{\beta}\cdot\bigoplus_{\wt{\omega}}(\wt{\pi}_{1}\wt{\otimes}\dots\wt{\otimes}\wt{\pi}_{k})_{\wt{\omega}},
			\end{equation}
			where $\wt{\omega}$ ranges over all the characters of $Z(\wt{G_{r}})$ satisfying \eqref{eqomegacompatible}, and $(\wt{\pi}_{1}\wt{\otimes}\dots\wt{\otimes}\wt{\pi}_{k})_{\wt{\omega}}$ is an irreducible representation of $\wt{G_{\gamma}}$ whose central character restricted to $Z(\wt{G_{r}})$ is $\wt{\omega}$, and $M_{\beta}:=n_{\beta}\cdot\prod_{i=1}^{k}m_{(r_{i})}$. Here each $m_{(r_{i})}$ is defined as in (2) by taking $\beta=(r_{i})$ correspondingly.
			
			\item \begin{itemize}
				\item For two genuine characters $\wt{\omega}, \wt{\omega}'$ of $Z(\wt{G_{r}})$ satisfying (\ref{eqomegacompatible}), there exists a character $\chi$ of $G_{\gamma}/G_{\gamma}^{(n)}$ such that $$(\wt{\pi}_{1}\wt{\otimes}\dots\wt{\otimes}\wt{\pi}_{k})_{\wt{\omega}'}\simeq (\wt{\pi}_{1}\wt{\otimes}\dots\wt{\otimes}\wt{\pi}_{k})_{\wt{\omega}}\cdot\chi.$$
				
				\item Let $\wt{\pi}_{i}'\in\mrirr_{\epsilon}(\wt{G_{\gamma_{i}}})$ for $i=1,\dots,k$ such that $(\wt{\pi}_{1},\dots,\wt{\pi}_{k};\wt{\omega})$ and $(\wt{\pi}_{1}',\dots,\wt{\pi}_{k}';\wt{\omega})$ satisfy \eqref{eqomegacompatible}, then
				$$(\wt{\pi}_{1}'\wt{\otimes}\dots\wt{\otimes}\wt{\pi}_{k}')_{\wt{\omega}}\simeq (\wt{\pi}_{1}\wt{\otimes}\dots\wt{\otimes}\wt{\pi}_{k})_{\wt{\omega}}$$
				if and only if
				$\wt{\pi}_{i}'\simeq\wt{\pi}_{i}\chi_{i}$ for a certain character $\chi_{i}$ of $G_{\gamma_{i}}/G_{\gamma_{i}}^{(n)}$ for each $i$.
				\item Every $\wt{\pi}\in\mrirr_{\epsilon}(\wt{G_{\gamma}})$ is of the form $(\wt{\pi}_{1}\wt{\otimes}\dots\wt{\otimes}\wt{\pi}_{k})_{\wt{\omega}}$ for certain $\wt{\pi}_{i}\in\mrirr_{\epsilon}(\wt{G_{\gamma_{i}}})$ for each $i$ and a compatible genuine character $\wt{\omega}$ of $Z(\wt{G_{r}})$.
			\end{itemize}
			
		\end{enumerate}
		
	\end{theorem}
	
	Thus for $\wt{\pi}_{i}\in\mrirr_{\epsilon}(\wt{G_{\gamma_{i}}})$ and a character $\wt{\omega}$ of $Z(\wt{G_{r}})$ satisfying \eqref{eqomegacompatible}, we define $(\wt{\pi}_{1}\wt{\otimes}\dots\wt{\otimes}\wt{\pi}_{k})_{\wt{\omega}}$ the \emph{metaplectic tensor product} with respect to $(\wt{\pi}_{1},\dots,\wt{\pi}_{k};\wt{\omega})$ as in the theorem. It is clear that $(\wt{\pi}_{1}\wt{\otimes}\dots\wt{\otimes}\wt{\pi}_{k})_{\wt{\omega}}\in\mcc_{\epsilon}(\wt{G_{\gamma}})$ if and only if $\wt{\pi}_{i}\in\mcc_{\epsilon}(\wt{G_{\gamma_{i}}})$ for each $i=1,\dots,k$, where $\mcc$ could be $\mrtemp$, $\mrsqrt$ and $\mrcusp$.
	
	For any positive integer $m$ we define $d_{F}(m):=m\cdot\abs{m}^{-1/2}$. So $d_{F}:\mbz_{>0}\rightarrow \mbq_{>0}$ is a multiplicative function. In particular $d_{F}(m)=m$ if $\mrgcd(m,p)=1$. 
	
	For a positive integer $r_{0}$, we define $d_{r_{0}}'=\mrgcd(n,r_{0})$ and $d_{r_{0}}=\mrgcd(n,2r_{0}c+r_{0}-1)$. We emphasize that $d_{r_{0}}$ depends on $n,r_{0}$ as well as $c$.
	
	We have the following basic lemma.
	
	\begin{lemma}[\cite{weil1974basic}*{Corollary of Proposition II.9, Proposition XIII.6}]\label{lemmaFmindex}
		
		Assume $m$ divides $n$. Then we have $[F^{\times}:F^{\times m}]=d_{F}(m)^{2}$, and moreover $d_{F}(m)$ is an integer.
		
	\end{lemma}
	
	In practice, it is useful to know the exact value of the multiplicities $n_{\beta}$, $m_{\beta}$, $M_{\beta}$ considered in the theorem. 
	
	\begin{proposition}\label{propmulcal}
		
		\begin{enumerate} 
			\item $$n_{\beta}=\frac{\prod_{i=1}^{k}[F^{\times}:F^{\times n/d_{r_{i}}'}]^{1/2}}{[F^{\times n/d_{r}}:F^{\times n}]^{1/2}}=\frac{\dF{n^{k}}}{\dF{d_{r}\cdot\prod_{i=1}^{k}d_{r_{i}}'}}.$$
			
			\item $$m_{\beta}=\frac{[F^{\times}:F^{\times n}]^{k/2}\cdot\prod_{i=1}^{k}[F^{\times}:F^{\times d_{r_{i}}'}]^{1/2}}{[F^{\times n/d_{r}}:F^{\times n}]^{1/2}}=\frac{\dF{n^{k}\cdot\prod_{i=1}^{k}d_{r_{i}}'}}{\dF{d_{r}}}.$$
			
			\item $$M_{\beta}=\frac{[F^{\times}:F^{\times n}]^{k}}{[F^{\times n/d_{r}}:F^{\times n}]^{1/2}\cdot\prod_{i=1}^{k}[F^{\times n/d_{r_{i}}}:F^{\times n}]^{1/2}}=\frac{\dF{n^{2k}}}{\dF{d_{r}\cdot\prod_{i=1}^{k}d_{r_{i}}}}.$$
		\end{enumerate}
		
	\end{proposition}
	
	\begin{proof}
		
		Using Lemma \ref{lemmaFmindex} we only need to prove the left part of each equation. We have
		$$n_{\beta}^{2}=[\wt{Z(G_{\beta})}:Z(\wt{G_{r}})\wt{Z(G_{\beta,\beta}^{(n)})}]=[\wt{Z(G_{\beta})}:\wt{Z(G_{\beta,\beta}^{(n)})}]/[Z(\wt{G_{r}}):Z(\wt{G_{r}})\cap \wt{Z(G_{\beta,\beta}^{(n)})}].$$
		By definition,
		$$[\wt{Z(G_{\beta})}:\wt{Z(G_{\beta,\beta}^{(n)})}]=[Z(G_{\beta}):Z(G_{\beta,\beta}^{(n)})]=\prod_{i=1}^{k}[F^{\times}:F^{\times n/d_{r_{i}}'}],$$
		and by \eqref{eqZwtGr} and \eqref{eqZwtGrintersection} we have
		$$[Z(\wt{G_{r}}):Z(\wt{G_{r}})\cap \wt{Z(G_{\beta,\beta}^{(n)})}]=[F^{\times n/d_{r}}:F^{\times n}].$$
		So indeed we proved (1). By definition 
		$$[\wt{G_{\gamma}}:\wt{Z(G_{\beta})}\wt{G_{\gamma,\beta}^{(n)}}]=[G_{\gamma}:Z(G_{\beta})G_{\gamma,\beta}^{(n)}]=\prod_{i=1}^{k}[G_{\gamma_{i}}:Z(G_{r_{i}})G_{\gamma_{i}}^{(n)}]=\prod_{i=1}^{k}[F^{\times}:F^{\times d_{r_{i}}'}],$$
		so using (1) we also proved (2). Finally (3) follows from (1) and (2).
		
	\end{proof}
	
	Finally the following corollary will be used later on.
	
	\begin{corollary}\label{corpirest}
		
		Let $\wt{\pi}_{i}\in\mrirr_{\epsilon}(\wt{G_{\gamma_{i}}})$ for $i=1,\dots,k$ and let $\wt{\omega}$ be a character of $Z(\wt{G_{r}})$ satisfying \eqref{eqomegacompatible}, then
		$$\prod_{i=1}^{k}\dF{d_{r_{i}}}\cdot(\wt{\pi}_{1}^{(n)}\wt{\otimes}\dots\wt{\otimes}\wt{\pi}_{k}^{(n)})\simeq \dF{d_{r}}\cdot(\wt{\pi}_{1}\wt{\otimes}\dots\wt{\otimes}\wt{\pi}_{k})_{\wt{\omega}}^{(n)}.$$
		
	\end{corollary}
	
	\begin{proof}
		
		First we notice that the restriction $(\wt{\pi}_{1}\wt{\otimes}\dots\wt{\otimes}\wt{\pi}_{k})_{\wt{\omega}}^{(n)}$ does not depend on the choice of $\wt{\omega}$, which follows from Theorem \ref{thmMTP}.(4). Then restricting \eqref{eqMTPdef} to $\wt{G_{\gamma,\beta}^{(n)}}$ and using the Mackey formula, we have 
		$$[\wt{G_{\gamma}}:\wt{G_{\gamma,\beta}^{(n)}}]\cdot(\wt{\pi}_{1}^{(n)}\wt{\otimes}\dots\wt{\otimes}\wt{\pi}_{k}^{(n)}) \simeq M_{\beta}\cdot[F^{\times n/d_{r}}:F^{\times n}]\cdot (\wt{\pi}_{1}\wt{\otimes}\dots\wt{\otimes}\wt{\pi}_{k})_{\wt{\omega}}^{(n)}.$$
		Notice that 
		$$[\wt{G_{\gamma}}:\wt{G_{\gamma,\beta}^{(n)}}]=[G_{\gamma}:G_{\gamma,\beta}^{(n)}]=[F:F^{\times n}]^{k},$$
		then our result follows from the Krull-Schmidt theorem and Proposition \ref{propmulcal}.(3).
		
	\end{proof}
	
	\section{Metaplectic correspondence}
	
	In this section we recall and update the metaplectic correspondence studied by Flicker and Kazhdan \cite{flicker1986metaplectic}.
	
	\subsection{Definition}\label{subsectionMCdefinition}
	
	Fix a character $\omega:Z(G_{r})\rightarrow\mbc^{\times}$ and a genuine character $\wt{\omega}:Z(\wt{G_{r}})\rightarrow\mbc^{\times}$ such that 
	\begin{equation}\label{eqomegawtomega}
		\wt{\omega}(\bs{s}(z^{n}))=\omega(z)\quad\text{for every}\ z\in Z(G_{r}).
	\end{equation}
	Notice that the existence of such $\wt{\omega}$ is equivalent to $\omega\rest_{\mu_{n}}=1$, and conversely $\omega$ is totally determined by $\wt{\omega}$.
	Consider the map
	$$G_{r}\rightarrow\wt{G_{r}},\quad x\mapsto x^{*}=u(x) \bs{s}(x)^{n},$$
	where $u(x)=\pm1$ is defined exactly as in \cite{flicker1986metaplectic}*{\S 4}. The key point is that this ``star" map is equivariant under $G_{r}$-conjugacy (\emph{cf.} \cite{kazhdan1986towards}*{Theorem 4.1}), thus it induces a map from the set of conjugacy classes of $\wt{G_{r}}$ to that of $G_{r}$.
	
	Let $\beta=(r_{1},\dots,r_{k})$ be a composition of $r$. An element $g\in G_{\beta}$ is called 
	\begin{itemize}
		\item  \emph{semi-simple} if it is diagonalizable as a matrix in $\mrgl_{r}(\overline{F})$ with $\overline{F}$ being the algebraic closure of $F$;
		\item \emph{semi-simple regular} if the roots of its characteristic polynomial are pairwise different;
		\item \emph{elliptic} if $F[X]/(P_{g}(X))$ is a field of degree $r$ over $F$, where $P_{g}(X)$ denotes the characteristic polynomial of $g$.
	\end{itemize}
	For $g=\mrdiag(g_{1},\dots,g_{k})\in G_{\beta}$ semi-simple and $\alpha_{i1},\dots,\alpha_{ir_{i}}$ the roots of the characteristic polynomial of $g_{i}$ for each $i$, we define
	$$D^{G_{\beta}}(g)=\prod_{i=1}^{k}\prod_{\substack{1\leq j<l\leq r_{i},\\ \alpha_{ij}\neq\alpha_{il}}}\frac{\abs{\alpha_{ij}-\alpha_{il}}^{2}}{\abs{\alpha_{ij}\alpha_{il}}}$$
	the Weyl discriminant of $g$, and
	$\Delta^{G_{\beta}}(g)=D^{G_{\beta}}(g)^{1/2}.$
	
	For $\pi\in\mrirr(G_{\beta})$ and $\wt{\pi}\in\mrirr_{\epsilon}(\wt{G_{\beta}})$, we say that $\wt{\pi}$ \emph{lifts to} $\pi$ (or $\pi$ is a \emph{lift} of $\wt{\pi}$) if the following relation is satisfied:
	\begin{equation}\label{eqmccorr}
		\Delta^{G_{\beta}}(x^{n})\cdot\theta_{\wt{\pi}}(x^{*})=\frac{1}{\dF{d_{r}}\abs{n}^{r/2}}\sum_{\{t\in G_{\beta} \mid t^{n}=x^{n}\}}\Delta^{G_{\beta}}(t)\cdot\theta_{\pi}(t)\cdot\epsilon(x^{*}/t^{*}),
	\end{equation}
	where $\theta_{\wt{\pi}}$ and $\theta_{\pi}$ denote the corresponding Harish-Chandra characters, and $x\in G_{\beta}$ such that $x^{n}$ is semi-simple regular. 
	
	\begin{remark}
		
		We mention the difference between \eqref{eqmccorr} and that in \cite{flicker1986metaplectic}. In \emph{ibid.} they assume the restriction of $\omega_{\wt{\pi}}$ to $Z(\wt{G_{r}})$ to be $\wt{\omega}$ and the restriction of $\omega_{\pi}$ to $Z(G_{r})$ to be $\omega$, whereas we release this restriction. Moreover the relation of Harish-Chandra characters in \emph{ibid.} is of the following form
		$$\Delta^{G_{\beta}}(x^{n})\cdot\theta_{\wt{\pi}}(x^{*})=\frac{n}{\dF{d_{r}}\abs{n}^{r/2}}\sum_{t,\wt{z}}\Delta^{G_{\beta}}(t)\cdot\theta_{\pi}(t)\cdot\wt{\omega}(\wt{z}),$$
		where the sum ranges over
		$\{t\in G_{\beta}/Z(G_{r}), \wt{z}\in Z(\wt{G_{r}})\mid t^{*}\wt{z}=x^{*}\}$.
		To verify that the two formulae coincide, we notice that for such a pair $(t,\wt{z})$ we have $\wt{z}\in Z(\wt{G_{r}})\cap\wt{G^{(n)}_{\beta}}$. Thus by \eqref{eqZwtGrintersection}, we may choose a representative $t$ such that $t^{n}=x^{n}$ and the corresponding $\wt{z}\in\mu_{n}$. Or in other words, the above set is in bijection with $\{t\in G_{\beta}/\mu_{n}I_{r}\mid t^{n}=x^{n}\}$. Since $\wt{\omega}(\wt{z})=\wt{\omega}(x^{*}/t^{*})=\epsilon(x^{*}/t^{*})$ and $\omega_{\pi}$ is trivial on $\mu_{n}I_{r}$, the two expressions coincide.
		
	\end{remark}
	
	\begin{remark}\label{remTrpitildesupp}
		
		Using \cite{flicker1986metaplectic}*{ Proposition 3}, the Harish-Chandra character $\theta_{\wt{\pi}}$ has support contained in $Z(\wt{G_{r}})\bigcup_{T\subset G_{\beta}}\wt{T^{(n)}}$, where $T$ in the union ranges over all the maximal torus of $G_{\beta}$ and $T^{(n)}:=\{x^{n}\mid x\in T\}$. So once we know $\wt{\pi}\rest_{Z(\wt{G_{r}})}$, the Harish-Chandra character of $\wt{\pi}$, or even $\wt{\pi}$ itself, is determined by $\pi$ via (\ref{eqmccorr}). 
		
	\end{remark}
	
	We write $\mrsqrt_{\wt{\omega}}(\wt{G_{\beta}})$ 
	and $\mrsqrt_{\omega}(G_{\beta})$ 
	as in \S \ref{subsectionnotation}. We denote by $\mrsqrt^{(n)}_{\omega}(G_{\beta})$ 
	the subset of $\mrsqrt_{\omega}(G_{\beta})$ 
	consisting of representations whose central character restricted to $\prod_{i=1}^{k}\mu_{n}$ is trivial. Here $\prod_{i=1}^{k}\mu_{n}$ is realized as a subgroup of $Z(G_{\beta})=\mrdiag(F^{\times} I_{r_{1}},\dots ,F^{\times} I_{r_{k}})$ via the block-diagonal embedding.
	
	\begin{remark}\label{remTrpisupp}
		
		Consider $x=\mrdiag(x_{1},\dots,x_{k})\in G_{\beta}$ such that $x^{n}$ is elliptic, then the $t$ in (\ref{eqmccorr}) ranges over $\mrdiag(x_{1}\zeta_{1},\dots,x_{k}\zeta_{k})$ with $\zeta_{i}\in\mu_{n}$. In this case for $\pi\in\mrsqrt^{(n)}_{\omega}(G_{\beta})$ the right hand side of (\ref{eqmccorr}) becomes $$\frac{n^{k}}{\dF{d_{r}}\abs{n}^{r/2}}\cdot \Delta^{G_{\beta}}(x)\cdot\theta_{\pi}(x).$$
		Since such $x$'s form a dense subset of the locus of elliptic elements in $G_{\beta}$, such $\pi$ is uniquely determined by $\wt{\pi}$ via (\ref{eqmccorr}).

	\end{remark}
	
	Let $\gamma_{i}$ be a composition of $r_{i}$ for each $i$ and let $\gamma=\gamma_{1}+\cdots+\gamma_{k}$ be a refinement of $\beta$ as before.
	
	\begin{proposition}\label{propmcMTPcomp}
		
		Assume $\wt{\pi}_{i}\in\mrirr_{\epsilon}(\wt{G_{\gamma_{i}}})$ and $\pi_{i}\in\mrirr(G_{\gamma_{i}})$ to be a lift of $\wt{\pi}_{i}$ for each $i=1,\dots,k$. Let $\pi=\pi_{1}\otimes\dots\otimes\pi_{k}\in\mrirr(G_{\gamma})$ and let $\wt{\pi}=(\wt{\pi}_{1}\wt{\otimes}\cdots\wt{\otimes}\wt{\pi}_{k})_{\wt{\omega}}\in\mrirr_{\epsilon}(\wt{G_{\gamma}})$ with a compatible character $\wt{\omega}$ of $Z(\wt{G_{r}})$ satisfying \eqref{eqomegacompatible}, then $\wt{\pi}$ lifts to $\pi$. 
		
	\end{proposition}
	
	\begin{proof}
		
		First by Corollary \ref{corpirest}, we have $$\dF{d_{r}}\cdot\wt{\pi}^{(n)}\simeq \prod_{i=1}^{k}\dF{d_{r_{i}}}\cdot(\wt{\pi}_{1}^{(n)}\wt{\otimes}\dots\wt{\otimes}\wt{\pi}_{k}^{(n)}).$$ In addition for $x=\mrdiag(x_{1},\dots,x_{k}), t=\mrdiag(t_{1},\dots,t_{k})\in G_{\beta}=G_{\gamma_{1}}\times\dots\times G_{\gamma_{k}}$, we have $$D^{G_{\gamma}}(x^{n})=\prod_{i=1}^{k}D^{G_{\gamma_{i}}}(x_{i}^{n})\quad\text{and}\quad D^{G_{\gamma}}(t)=\prod_{i=1}^{k}D^{G_{\gamma_{i}}}(t_{i}).$$ 
		By definition we have $x^{*},t^{*}\in\wt{G^{(n)}_{\gamma,\beta}}$ and $x_{i}^{*}, t_{i}^{*}\in\wt{G^{(n)}_{\gamma_{i}}}$ for each $i$. We also regard each $x_{i}^{*}$ and $t_{i}^{*}$ as elements in $\wt{G_{\gamma}}$ via the corresponding block-diagonal embedding.
		
		\begin{lemma}
			
			We have $x^{*}=x_{1}^{*}x_{2}^{*}\dots x_{k}^{*}$ and $t^{*}=t_{1}^{*}t_{2}^{*}\dots t_{k}^{*}.$
			
		\end{lemma}
		
		\begin{proof}
			
			We only need to prove the first equation and to consider the $k=2$ and $\gamma=\beta$ case. By definition and \eqref{eqBLS} we have $$\bs{s}(x)=\bs{s}(\mrdiag(x_{1},x_{2}))=\bs{s}(x_{2})\bs{s}(x_{1})=\bs{s}(x_{1})\bs{s}(x_{2})(\mrdet(x_{1}),\mrdet(x_{2}))_{n}^{-1}.$$ Thus by direct calculation 
			\begin{align*}
				x^{*}&=\bs{s}(\mrdiag(x_{1},x_{2}))^{n}u(x)=\bs{s}(x_{1})^{n}\bs{s}(x_{2})^{n}(\mrdet(x_{1}),\mrdet(x_{2}))_{n}^{-n(n+1)/2}u(x)\\
				&=x_{1}^{*}x_{2}^{*}(\mrdet(x_{1}),\mrdet(x_{2}))_{n}^{-n(n+1)/2}u(x)/u(x_{1})u(x_{2})=x_{1}^{*}x_{2}^{*}.
			\end{align*}
			where for the last step we use the formula 
			$$u(x)/u(x_{1})u(x_{2})=\begin{cases}
				1 \quad &\text{if}\ n\ \text{is odd};\\ 	(\mrdet(x_{1}),\mrdet(x_{2}))_{2} \quad& \text{if}\ n\ \text{is even}.
			\end{cases}$$
			given in \cite{flicker1986metaplectic}*{Section 4}.
			
		\end{proof}
		
		Using the facts above, we have
		\begin{align*}
			&\Delta^{G_{\gamma}}(x^{n})\cdot\theta_{\wt{\pi}}(x^{*})\\
			&=\frac{\prod_{i=1}^{k}\dF{d_{r_{i}}}}{\dF{d_{r}}}\prod_{i=1}^{k}\Delta^{G_{\gamma_{i}}}(x_{i}^{n})\cdot\theta_{\wt{\pi}_{i}}(x_{i}^{*})\\
			&=\frac{\prod_{i=1}^{k}\dF{d_{r_{i}}}}{\dF{d_{r}}}\prod_{i=1}^{k}\frac{1}{\dF{d_{r_{i}}}\abs{n}^{r_{i}/2}}\sum_{\{t_{i}\in G_{\gamma_{i}}\mid t_{i}^{n}=x_{i}^{n}\}}\Delta^{G_{\gamma_{i}}}(t_{i})\cdot\theta_{\pi_{i}}(t_{i})\cdot\epsilon(x_{i}^{*}/t_{i}^{*})\\
			&=\frac{1}{\dF{d_{r}}\abs{n}^{r/2}}\sum_{\{t\in G_{\gamma}\mid t^{n}=x^{n}\}}\Delta^{G_{\gamma}}(t)\cdot\theta_{\pi}(t)\cdot\epsilon(x^{*}/t^{*}).
		\end{align*}

	\end{proof}
	
	Using the above proposition, we may consider the metaplectic correspondence for essentially square integrable representations of $\wt{G_{\beta}}$ and $G_{\beta}$.
	
	\begin{proposition}\label{propmclift}
		
		For any $\pi\in\mrsqrt^{(n)}_{\omega}(G_{\beta})$, there exists a unique $\wt{\pi}\in\mrsqrt_{\wt{\omega}}(\wt{G_{\beta}})$ which lifts to $\pi$, and conversely any $\wt{\pi}\in\mrsqrt_{\wt{\omega}}(\wt{G_{\beta}})$ lifts to a unique $\pi\in\mrsqrt^{(n)}_{\omega}(G_{\beta})$.
		
	\end{proposition}
	
	\begin{proof}
		
		The $k=1$ case is studied by Flicker-Kazhdan based on the simple trace formula (\cite{flicker1986metaplectic}*{ Theorem 26.1}), which will be our starting point. Now we consider the general case. 
		
		For $\pi\in\mrsqrt^{(n)}_{\omega}(G_{\beta})$ 
		, there exists a unique $\pi_{i}\in\mrsqrt(G_{r_{i}})$ 
		for each $i$ such that $\pi=\pi_{1}\otimes\dots\otimes\pi_{k}$. Similarly for $\wt{\pi}\in\mrsqrt_{\wt{\omega}}(\wt{G_{\beta}})$, there exists  $\wt{\pi_{i}}\in\mrsqrt_{\epsilon}(\wt{G_{r_{i}}})$ for each $i$
		such that $\wt{\pi}=(\wt{\pi}_{1}\otimes\dots\otimes\wt{\pi}_{k})_{\wt{\omega}}$, where $\wt{\omega}$ is a character of $Z(\wt{G_{r}})$ satisfying \eqref{eqomegacompatible}. 
		
		So if $\pi=\pi_{1}\otimes\dots\otimes\pi_{k}\in\mrsqrt^{(n)}_{\omega}(G_{\beta})$ is given, then for each $i$ we choose $\wt{\pi}_{i}\in\mrsqrt_{\epsilon}(\wt{G_{r_{i}}})$ that lifts to $\pi_{i}$. Using Proposition \ref{propmcMTPcomp} the metaplectic tensor product $\wt{\pi}=(\wt{\pi}_{1}\wt{\otimes}\dots\wt{\otimes}\wt{\pi}_{k})_{\wt{\omega}}$ lifts to $\pi$. Conversely for any given $\wt{\pi}=(\wt{\pi}_{1}\wt{\otimes}\dots\wt{\otimes}\wt{\pi}_{k})_{\wt{\omega}}$, we let $\pi_{i}$ be the lift of $\wt{\pi}_{i}$. Still by Proposition \ref{propmcMTPcomp} $\pi=\pi_{1}\otimes\dots\otimes\pi_{k}$ is a lift of $\wt{\pi}$. Finally the uniqueness follows from the formula (\ref{eqmccorr}) and the linear independence of Harish-Chandra characters, see Remark \ref{remTrpitildesupp} and Remark \ref{remTrpisupp}.
		
	\end{proof}
	
	\begin{definition}
		
		By the above proposition, we define the \emph{metaplectic correspondence}
		$$\mrmc:\mrsqrt_{\wt{\omega}}(\wt{G_{\beta}})\ 
		\longrightarrow \mrsqrt_{\omega}(G_{\beta})$$\ 
		as an injection mapping each representation to its lift, whose image is $ \mrsqrt^{(n)}_{\omega}(G_{\beta})$. 
		
	\end{definition}
	
	\begin{remark}
		
		We note that the generalization of metaplectic correspondence to Levi subgroups could have been done in \cite{flicker1986metaplectic}*{\S 26.2}. However there exists a mistake in their study of metaplectic tensor product during the construction, as already been pointed out by other authors. In \cite{mezo2002comparisons} and \cite{mezo2004metaplectic}, Mezo somehow managed to fix the mistake of Flicker and Kazhdan, however at least for general $r$ and $n$ he didn't write down all the details, especially it is unclear to the author why the scalar (i.e. $1/(\dF{d_{r}}\abs{n}^{r/2})$) in the formula \eqref{eqmccorr} of Harish-Chandra characters should be the correct one. Thus, our contribution here is to use the refined ``metaplectic tensor product" studied in \cite{kaplan2022classification} to fix the mistake above, which in particular gives the correct scalar in \eqref{eqmccorr}.
		
	\end{remark}
	
	\subsection{The case of a maximal split torus}
	
	We digress for a moment to study an important example. Hopefully it provides a better understanding of \eqref{eqmccorr} and Proposition \ref{propmclift}. This subsection is independent, so the readers are free to skip it.
	
	In this subsection we assume $\beta=(1,1,\dots,1)$, or in other words, $G_{\beta}$ is the diagonal torus of $G_{r}$. We notice that in this special case we have $Z(G_{\beta})=G_{\beta}$ and $G_{\beta,\beta}^{(n)}=F^{\times n}\times\dots\times F^{\times n}$. In particular, both $G_{\beta}$ and $\wt{G_{\beta,\beta}^{(n)}}$ are abelian.
	
	By definition every irreducible representation $\chi$ of $G_{\beta}$ is indeed a character. 
	Correspondingly, let $\wt{\chi}$ be an irreducible genuine representation of $\wt{G_{\beta}}$. By Proposition \ref{thmMTP}.(2), we have $\wt{\chi}^{(n)}:=\wt{\chi}\rest_{\wt{G_{\beta,\beta}^{(n)}}}=n_{\beta}\cdot \wt{\tau}$, where $\wt{\tau}$ is a genuine character of $\wt{G_{\beta,\beta}^{(n)}}$. 
	
	\begin{proposition}
		
		Let $\chi\in\mrirr(G_{\beta})$ and let $\wt{\chi}\in\mrirr_{\epsilon}(\wt{G_{\beta}})$. Then $\wt{\chi}$ lifts to $\chi$ if and only if \begin{enumerate}
			\item $\chi$ is trivial on $\prod_{i=1}^{r}\mu_{n}\subset G_{\beta}$ (via the diagonal embedding).
			\item  $\wt{\tau}$ is the genuine character satisfying $\wt{\tau}(x^{*})=\chi(x)$ for $x\in G_{\beta}$.
		\end{enumerate} 
		
	\end{proposition}
	
	\begin{proof}
		
		Taking $\wt{\pi}=\wt{\chi}$ and $\pi=\chi$, the left-hand side of \eqref{eqmccorr} becomes 
		$$n_{\beta}\cdot\wt{\tau}(x^{*}).$$
		Taking $\pi=\chi$, the right-hand side of \eqref{eqmccorr} becomes
		\begin{align*}
			\frac{1}{\dF{d_{r}}\abs{n}^{r/2}}\sum_{\{t\in G_{\beta}\mid t^{n}=x^{n}\}}\chi(t)\cdot\epsilon(x^{*}/t^{*})=\frac{1}{\dF{d_{r}}\abs{n}^{r/2}}\sum_{\zeta\in\prod_{i=1}^{r}\mu_{n}}\chi(x\zeta).
		\end{align*}
		So it equals $0$ unless $\chi$ is trivial on $\prod_{i=1}^{r}\mu_{n}$, and in this case it equals 
		$$\frac{\dF{n^{r}}}{\dF{d_{r}}}\cdot\chi(x).$$
		Finally using Proposition \ref{propmulcal}.(2) we have $n_{\beta}=\dF{n^{r}}/\dF{d_{r}}$, which finishes the proof.
		
	\end{proof}
	
	Thus for $\chi\in\mrirr(G_{\beta})$ trivial on $\prod_{i=1}^{r}\mu_{n}$ and $\wt{\chi}\in\mrirr_{\epsilon}(\wt{G_{\beta}})$, we have $\wt{\chi}$ lifts to $\chi$ if and only if $\wt{\chi}^{(n)}$ contains $\wt{\tau}$ as in the proposition. If we also fix $\wt{\omega}$ as the central character of $\wt{\chi}$ restricted to $Z(\wt{G_{r}})$, then $\wt{\chi}$ itself is uniquely determined by $\wt{\tau}$ and $\wt{\omega}$ (\emph{cf.} Theorem \ref{thmMTP}). This justifies Proposition \ref{propmclift} in this special case.
	
	\subsection{Compatibility with parabolic induction and Jacquet module}\label{subsectioncomppindjac}

	As before we let $\beta$, $\gamma$ be compositions of $r$, such that $\gamma$ is a refinement of $\beta$. Let $P_{\beta,\gamma}$ be a parabolic subgroup of $G_{\beta}$ having the Levi subgroup $G_{\gamma}$, and let $N_{\beta,\gamma}$ be the corresponding unipotent radical. By \cite{moeglin1995spectral}*{Appendix I} there exists a canonical splitting $\bs{s}_{N_{\beta,\gamma}}:N_{\beta,\gamma}\rightarrow \wt{G_{\beta}}$ as a group homomorphism that is $P_{\beta,\gamma}$-equivariant. Using this splitting we also realize $N_{\beta,\gamma}$ as a subgroup of $\wt{G_{\beta}}$. We consider the following normalized parabolic induction and normalized Jacquet module with respect to the decomposition $P_{\beta,\gamma}=G_{\gamma}N_{\beta,\gamma}$ and $\wt{P_{\beta,\gamma}}=\wt{G_{\gamma}}N_{\beta,\gamma}$:
	$$i_{\gamma,\beta}:=i_{N_{\beta,\gamma},1}:\mrrep(G_{\gamma})\rightarrow\mrrep(G_{\beta})\quad(\text{or}\ \mrrep(\wt{G_{\gamma}})\rightarrow\mrrep(\wt{G_{\beta}})) $$
	$$r_{\beta,\gamma}:=r_{N_{\beta,\gamma},1}:\mrrep(G_{\beta})\rightarrow\mrrep(G_{\gamma})\quad(\text{or}\ \mrrep(\wt{G_{\beta}})\rightarrow\mrrep(\wt{G_{\gamma}}))$$
	in the sense of \cite{bernstein1977induced}*{\S 2.3}.
	
	The metaplectic correspondence is compatible with the parabolic induction and Jacquet functor in the following sense.
	
	\begin{proposition}\label{propmcindres}
		
		\begin{enumerate}
			\item For $\rho\in\mrirr(G_{\gamma})$ as a lift of $\wt{\rho}\in\mrirr_{\epsilon}(\wt{G_{\gamma}})$, 
			we have that $i_{\gamma,\beta}(\wt{\rho})$ and $i_{\gamma,\beta}(\rho)$ satisfy the equation \eqref{eqmccorr}. In other words, $i_{\gamma,\beta}(\wt{\rho})$ lifts to $i_{\gamma,\beta}(\rho)$ in the general sense (without sticking to irreducible representations).
			
			\item Let $\wt{\pi}\in\mrsqrt_{\wt{\omega}}(\wt{G_{\beta}})$ and let $\pi\in\mrsqrt_{\omega}^{(n)}(G_{\beta})$ be a lift of $\wt{\pi}$. If $r_{\beta,\gamma}(\pi)\in\mrsqrt^{(n)}_{\omega}(G_{\gamma})$, then $r_{\beta,\gamma}(\wt{\pi})\in\mrsqrt(\wt{G_{\gamma}})$ (thus in particular it is irreducible) which lifts to $r_{\beta,\gamma}(\pi)$; otherwise $r_{\beta,\gamma}(\wt{\pi})=0$.
		\end{enumerate}
		
	\end{proposition}
	
	\begin{proof}
		
		The statement (1) has already been used in \cite{flicker1986metaplectic}*{Proposition 26.2} with a one-sentence explanation, and here we give a rather detailed proof for completeness. For $x\in G_{\beta}$ such that $x^{n}$ is semi-simple regular, we have
		\begin{align*}
			&\quad \Delta^{G_{\beta}}(x^{n})\cdot\theta_{i_{\beta,\eta}(\wt{\rho})}(x^{*})\\
			&=\sum_{\{\wt{g}\in\wt{G_{\beta}}/\wt{P_{\beta,\gamma}}\mid\wt{g}^{-1}x^{*}\wt{g}\in\wt{G_{\gamma}}\}}\Delta^{G_{\gamma}}(g^{-1}x^{n}g)\cdot\theta_{\wt{\rho}}(\wt{g}^{-1}x^{*}\wt{g})\\
			&=\sum_{\{g\in G_{\beta}/P_{\beta,\gamma}\mid g^{-1}x^{n}g\in G_{\gamma}\}}\Delta^{G_{\gamma}}(g^{-1}x^{n}g)\cdot\theta_{\wt{\rho}}((g^{-1}xg)^{*})\\
			&=\frac{1}{\dF{d_{r}}\abs{n}^{r/2}}\sum_{\{g\in G_{\beta}/P_{\beta,\gamma}\mid g^{-1}x^{n}g\in G_{\gamma}\}}\sum_{\{t\mid t^{n}=x^{n}\}}\Delta^{G_{\gamma}}(g^{-1}tg)\cdot\theta_{\rho}(g^{-1}tg)\cdot\epsilon((g^{-1}xg)^{*}/(g^{-1}tg)^{*})\\
			&=\frac{1}{\dF{d_{r}}\abs{n}^{r/2}}\sum_{\{t\mid t^{n}=x^{n}\}}\sum_{\{g\in G_{\beta}/P_{\beta,\gamma}\mid g^{-1}tg\in G_{\gamma}\}}\Delta^{G_{\gamma}}(g^{-1}tg)\cdot\theta_{\rho}(g^{-1}tg)\cdot\epsilon(x^{*}/t^{*})\\
			&=\frac{1}{\dF{d_{r}}\abs{n}^{r/2}}\sum_{\{t\mid t^{n}=x^{n}\}}\Delta^{G_{\beta}}(t)\cdot\theta_{i_{\gamma,\beta}(\rho)}(t)\cdot\epsilon(x^{*}/t^{*}).
		\end{align*}
		Here we use \cite{van1972computation}*{Theorem 3}, which also works for covering groups with the same argument, for the first and final equations. We use the fact $\bs{s}(g)^{-1}x^{*}\bs{s}(g)=(g^{-1}xg)^{*}$ for the second equation and (\ref{eqmccorr}) for the third equation. Finally $ g^{-1}tg\in G_{\beta}$ is equivalent to  $g^{-1}t^{n}g\in G_{\beta}$, since $t^{n}=x^{n}$ is semi-simple regular. We use this statement and the fact that $(g^{-1}xg)^{*}/(g^{-1}tg)^{*}=\bs{s}(g)^{-1}(x^{*}/t^{*})\bs{s}(g)=x^{*}/t^{*}$ for the fourth equation. 
		
		For the statement (2), we write $\gamma=(r_{1}',\dots,r_{l}')$. We first recall a result of Casselman. 
		Let $t=\mrdiag(t_{1},\dots,t_{l})\in G_{\gamma}$, such that $t_{j}$ is an element in $G_{r_{j}'}$ with $t_{j}^{n}$ being elliptic for each $j$, and moreover
		\item\begin{equation}\label{eqcondtidi+1}
			\abs{\mrdet(t_{j})}^{1/r_{j}'}/\abs{\mrdet(t_{j+1})}^{1/r_{j+1}'}>1\quad \text{for}\ j=1,\dots,l-1.
		\end{equation} 
		Then we have 
		(\emph{cf.} \cite{casselman1977characters}*{Theorem 5.2} and \cite{flicker1986metaplectic}*{Theorem 14}) 
		\begin{equation}\label{eqHCcharJac}
			\Delta^{G_{\beta}}(t^{n})\cdot\theta_{\wt{\pi}}(t^{*})=\Delta^{G_{\gamma}}(t^{n})\cdot\theta_{r_{\beta,\gamma}(\wt{\pi})}(t^{*})\quad\text{and}\quad\Delta^{G_{\beta}}(t)\cdot\theta_{\pi}(t)=\Delta^{G_{\gamma}}(t)\cdot\theta_{r_{\beta,\gamma}(\pi)}(t).
		\end{equation}
		
		To continue, we need the following technical lemma. 
		
		\begin{lemma}\label{lemmalinearrelasqrtrep}
			
			Let $\gamma=(r_{1}',\dots,r_{l}')$ be a composition of $r$ as above and let $\wt{\tau}_{1},\dots,\wt{\tau}_{s'}\in\mrsqrt_{\epsilon}(\wt{G^{(n)}_{\gamma,(r)}})$ with $G^{(n)}_{\gamma,(r)}:=G_{r_{1}'}^{(n)}\times\dots\times G_{r_{l}'}^{(n)}$. If for certain $c_{i}\in\mbc$ the equation
			\begin{equation}\label{eqTrtau'ici}
				\sum_{i=1}^{s'}c_{i}\cdot\theta_{\wt{\tau}_{i}}(t^{*})=0
			\end{equation}
			holds for any $t$ as above, then we have 
			$$\sum_{i=1}^{s'}c_{i}\cdot\theta_{\wt{\tau}_{i}}=0.$$
		\end{lemma}
		
		\begin{proof}
			
			First we need the following basic lemma.
			
			\begin{lemma}\label{lemmasigmalinearindep}
				
				Let $\wt{\sigma}_{1},\dots,\wt{\sigma}_{u}\in\mrsqrt_{\epsilon}(\wt{G_{r'}^{(n)}})$ be pairwise non-isomorphic representations and let $N>0$, then there exist elliptic elements $g_{1},\dots,g_{u}\in G_{r'}^{(n)}$  satisfying $\abs{\mrdet(g_{j})}<1/N$ for each $j=1,\dots, u$, such that the matrix $(\theta_{\wt{\sigma}_{k}}(\zeta_{j}\bs{s}(g_{j})))_{1\leq j,k\leq u}$ is invertible for any $\zeta_{j}\in\mu_{n}$.
				
			\end{lemma}
			
			\begin{proof}
				
				The lemma follows from the linear independence of $\theta_{\wt{\sigma}_{k}}$ on elliptic locus if we don't impose any restriction on $\abs{\mrdet(g_{j})}$. In general, we just need to replace $g_{j}$ by $g_{j}\varpi_{F}^{m}$ for $m$ large enough.
				
			\end{proof}
			
			We remark that each $\wt{\tau}_{i}$, as an essentially square integrable representation, is determined by the restriction of its Harish-Chandra character to the elliptic locus. We write $\wt{\tau}_{i}=\wt{\tau}_{i1}\wt{\otimes}\dots\wt{\otimes}\wt{\tau}_{il}$, where $\wt{\tau}_{ij}\in\mrsqrt_{\epsilon}(\wt{G_{r_{j}'}^{(n)}})$ for $j=1,\dots,l$. 
			
			We prove Lemma \ref{lemmalinearrelasqrtrep} by induction on $l$. The $l=1$ case is clear. Assume the lemma holds for $l-1$. We consider the division $\{1,\dots,s'\}=I_{1}\sqcup\dots\sqcup I_{u}$, such that $\wt{\tau}_{i_{1}l}\simeq\wt{\tau}_{i_{2}l}$ if and only if $i_{1},i_{2}\in I_{k}$ for some $k=1,\dots,u$. And we denote by $\wt{\tau}_{I_{k}l}$ the isomorphism class of $\wt{\tau}_{il}$ for $i\in I_{k}$. We write $\wt{\tau}_{i}'=\wt{\tau}_{i1}\wt{\otimes}\dots\wt{\otimes}\wt{\tau}_{il-1}$ and $\wt{\tau}_{i}=\wt{\tau}_{i}'\wt{\otimes}\wt{\tau}_{il}$ for each $i=1,\dots, s'$. For $t=\mrdiag(t_{1},\dots,t_{l})$ as above, we write $t'=\mrdiag(t_{1},\dots,t_{l-1})$. By \eqref{eqTrtau'ici} we have
			\begin{equation}\label{eqTrtau'icil}
				\sum_{k=1}^{u}\theta_{\wt{\tau}_{I_{k}l}}(t_{l}^{*})\sum_{i\in I_{k}}c_{i}\cdot\theta_{\wt{\tau}'_{i}}(t'^{*})=0.
			\end{equation}
			Using Lemma \ref{lemmasigmalinearindep} for $\wt{\sigma}_{k}=\wt{\tau}_{I_{k}l}$, we may choose $u$'s different $t_{l}$ with the corresponding $t_{l}^{*}$ denoted by $\wt{g}_{1},\dots,\wt{g}_{u}$, such that the matrix $(\theta_{\wt{\tau}_{I_{k}l}}(\wt{g}_{j}))_{1\leq j,k\leq u}$ is invertible. Plugging $t_{l}^{*}=\wt{g}_{j}$ into \eqref{eqTrtau'icil} for $j=1,\cdots, u$ we get
			$$\sum_{i\in I_{k}}c_{i}\cdot\theta_{\wt{\tau}'_{i}}(t'^{*})=0,\quad k=1,\dots,u.$$
			Using the induction hypothesis and the definition of $I_{k}$, we have
			
			$$\sum_{i\in I_{k}}c_{i}\cdot\theta_{\wt{\tau}_{i}'}=0\quad\text{and}\quad\sum_{i\in I_{k}}c_{i}\cdot\theta_{\wt{\tau}_{i}}=(\sum_{i\in I_{k}}c_{i}\cdot \theta_{\wt{\tau}_{i}'})\wt{\otimes} \theta_{\wt{\tau}_{I_{k}l}}=0,\quad k=1,\dots,u.$$
			Summing over $k$ we finish the proof.
			
		\end{proof}

		We finish the proof of Proposition \ref{propmcindres}.(2). We consider $\pi$, $\wt{\pi}$ as in the statement. 
		
		If $r_{\beta,\gamma}(\pi)\in\mrsqrt^{(n)}_{\omega}(G_{\gamma})$, then we let $\wt{\pi}'$ be the unique essentially square integrable representation of $\wt{G_{\gamma}}$ satisfying $\mrmc(\wt{\pi}')=r_{\beta,\gamma}(\pi)$ and having the central character $\wt{\omega}$ restricted to $Z(\wt{G_{r}})$. Then using \eqref{eqmccorr} and \eqref{eqHCcharJac} with a similar argument to that in (1), for $t$ as above we have
		\begin{equation}\label{eqTrpi'pii'}
			\theta_{\wt{\pi}'}(t^{*})=\theta_{r_{\beta,\gamma}(\wt{\pi})}(t^{*})=\sum_{i=1}^{s}\theta_{\wt{\pi}_{i}'}(t^{*}),
		\end{equation}
		where $\wt{\pi}_{1}'$,\dots,$\wt{\pi}_{s}'$ are all the irreducible subquotients of $r_{\beta,\gamma}(\wt{\pi}$). Considering the restriction of each $\wt{\pi}_{i}'$ and $\wt{\pi}'$ to $\wt{G_{\gamma,(r)}^{(n)}}$, using Theorem \ref{thmMTP}.(2) and Lemma \ref{lemmalinearrelasqrtrep}, and comparing the number of irreducible components (notice that the total multiplicities of each $\wt{\pi}_{i}'\rest_{\wt{G_{\gamma,(r)}^{(n)}}}$ and $\wt{\pi}'\rest_{\wt{G_{\gamma,(r)}^{(n)}}}$are the same), we must have $s=1$ and $\wt{\pi}'\rest_{\wt{G_{\gamma,(r)}^{(n)}}}\simeq \wt{\pi}_{1}'\rest_{\wt{G_{\gamma,(r)}^{(n)}}}$. Comparing the central character we also have $\omega_{\wt{\pi}'}\rest_{Z(\wt{G_{r}})}=\omega_{\wt{\pi}_{1}'}\rest_{Z(\wt{G_{r}})}$. By Remark \ref{remTrpitildesupp} or Theorem \ref{thmMTP} we have $\wt{\pi}'\simeq \wt{\pi}_{1}'=r_{\beta,\gamma}(\wt{\pi})$, so $\mrmc(r_{\beta,\gamma}(\wt{\pi}))=r_{\beta,\gamma}(\pi)$.
		
		
		Finally we consider the case $r_{\beta,\gamma}(\pi)\notin\mrsqrt_{\omega}^{(n)}(G_{\gamma})$. 
		
		\begin{lemma}\label{lemmaMCnulllift}
			
			For $\pi'=\pi_{1}'\otimes\dots\otimes\pi_{l}'\in\mrsqrt_{\omega}(G_{\gamma})-\mrsqrt^{(n)}_{\omega}(G_{\gamma})$ and $x=\mrdiag(x_{1},\dots,x_{l})\in G_{\gamma}$ such that $x_{i}^{n}$ is elliptic for any $i=1,\dots,l$, we have 
			$$\prod_{i=1}^{l}\sum_{\{t_{i}\mid t_{i}^{n}=x_{i}^{n}\}}\Delta^{G_{r_{i}'}}(t_{i})\cdot\theta_{\pi_{i}'}(t_{i})\cdot\epsilon(x_{i}^{*}/t_{i}^{*})=0.$$ 
			As a result $\wt{\pi}'=0$ may be regard as the ``null representation" that lifts to $\pi'$. 
			
		\end{lemma}
		
		\begin{proof}
			
			Since $x_{i}$ is elliptic for $i=1,\dots,l$ for each $t_{i}$ in the sum there exists $\zeta_{i}\in\mu_{n}$ such that $x_{i}=t_{i}\zeta_{i}$. Thus $D^{G_{r_{i}'}}(x_{i})=D^{G_{r_{i}'}}(t_{i})$ and $\epsilon(x_{i}^{*}/t_{i}^{*})=1$, and we only need to prove that
			$$\prod_{i=1}^{l}\bigg(\sum_{\zeta_{i}\in\mu_{n}}\omega_{\pi_{i}'}(\zeta_{i})\bigg)\cdot\theta_{\pi_{i}'}(x_{i})=0.$$
			The lemma follows from the fact that $\omega_{\pi'_{i}}\rest_{\mu_{n}}$ is not trivial for some $i$.	
			
		\end{proof}
		
		Using this lemma, \eqref{eqmccorr} and \eqref{eqHCcharJac} and a similar argument to the previous case, we have \begin{equation}
			\theta_{r_{\beta,\gamma}(\wt{\pi})}(t^{*})=\sum_{i=1}^{s}\theta_{\wt{\pi}_{i}'}(t^{*})=0,
		\end{equation}
		where $\wt{\pi}_{1}'$,\dots,$\wt{\pi}_{s}'$ are all the irreducible subquotients of $r_{\beta,\gamma}(\wt{\pi}$) and $t$ is considered as above. Still considering the restriction of each $\wt{\pi}_{i}'$ to $\wt{G_{\gamma,(r)}^{(n)}}$, and using Theorem \ref{thmMTP}.(2) and Lemma \ref{lemmalinearrelasqrtrep}, we have $s=0$ and $r_{\beta,\gamma}(\wt{\pi})=0$.
		
	\end{proof}
	
	\subsection{Classification of essentially square integrable representations}
	
	In this part we classify $\mrsqrt_{\epsilon}(\wt{G_{r}})$ as well as its image under the metaplectic correspondence.
	We let $\nu=\abs{\mrdet(\cdot)}$ be a character of $G_{r'}$ for each positive integer $r'$. For a composition  $\beta=(r_{1},\cdots,r_{k})$ of $r$, representations $\wt{\pi}_{i}\in\mrirr_{\epsilon}(\wt{G_{r_{i}}})$ for $i=1,\dots, k$ and a genuine character $\wt{\omega}$ of $Z(\wt{G_{r}})$ satisfying \eqref{eqomegacompatible}, we define the following Bernstein-Zelevinsky product 
	$$(\wt{\pi}_{1}\wt{\times}\dots\wt{\times}\wt{\pi}_{k})_{\wt{\omega}}:=i_{\beta,(r)}((\wt{\pi}_{1}\wt{\otimes}\dots\wt{\otimes}\wt{\pi}_{k})_{\wt{\omega}}).$$ 
	First we focus on a result which was previously studied in \cite{flicker1986metaplectic}*{Lemma 27.1}. But unfortunately both the statement and proof in \emph{loc. cit.} are flawed. So we give a correct one.
	
	\begin{proposition}\label{propvaluesrho}
		
		\begin{enumerate}
			\item Let $\wt{\rho}_{1}\in\mrcusp_{\epsilon}(\wt{G_{r_{1}}})$ and $\wt{\rho}_{2}\in\mrcusp_{\epsilon}(\wt{G_{r_{2}}})$. Then $(\wt{\rho}_{1}\wt{\times}\wt{\rho}_{2})_{\wt{\omega}}$ is reducible only if $r_{1}=r_{2}$, and $(\wt{\rho}_{2}\wt{\times}\wt{\rho}_{1})_{\wt{\omega}}$ is isomorphic to $(\wt{\rho}_{1}\wt{\times}\wt{\rho}_{2})_{\wt{\omega}}$ twisted by an unramified character of $G_{r_{1}+r_{2}}$.
			
			\item Let $\wt{\rho}\in\mrcusp_{\epsilon}(\wt{G_{r_{0}}})$. Then there exists a unique positive real number $s(\wt{\rho})$ such
			that $(\wt{\rho}\wt{\times}\wt{\rho}\nu^{s(\wt{\rho})})_{\wt{\omega}}$ is not irreducible. Moreover, $(\wt{\rho}\wt{\times}\wt{\rho}\nu^{s})_{\wt{\omega}}$ is not irreducible for some $s\in\mbr$ if and only if $s=\pm s(\wt{\rho})$.
			
			\item  In (2), we let $\{\rho\nu^{a},\rho\nu^{a+1},\dots,\rho\nu^{b}\}$ be the cuspidal support of $\mrmc(\wt{\rho})$, where $m=b-a+1$ is a positive integer and $\rho\in\mrcusp(G_{r_{0}/m})$. Then $s(\wt{\rho})=m/n$.
			
		\end{enumerate} 
		In (1) and (2), $\wt{\omega}$ are compatible genuine characters of $Z(\wt{G_{r_{1}+r_{2}}})$ and $Z(\wt{G_{2r_{0}}})$ respectively, such that the corresponding metaplectic tensor products make sense.
		
	\end{proposition}
	
	\begin{proof}
		
		The statement (1) and (2) have been proved in \cite{kaplan2022classification}*{Proposition 6.10}. So we focus on statement (3). We let $\beta_{0}=(r_{0},r_{0})$ be a composition of $2r_{0}$. The representation   $\wt{\pi}=(\wt{\rho}\wt{\times}\wt{\rho}\nu^{s(\wt{\rho})})_{\wt{\omega}}$ is reducible and of length 2, and moreover its Jacquet module $r_{(2r_{0}),\beta_{0}}(\wt{\pi})$ consists of $(\wt{\rho}\wt{\otimes}\wt{\rho}\nu^{s(\wt{\rho})})_{\wt{\omega}}$ and $(\wt{\rho}\nu^{s(\wt{\rho})}\wt{\otimes}\wt{\rho})_{\wt{\omega}}$ as its subquotient. Using Casselman's criterion for discrete series (\emph{cf.} \cite{ban2013langlands}*{ Theorem 3.4}), there exists an essentially square integrable subquotient $\wt{\pi}'$ of $\wt{\pi}$. Thus $r_{(2r_{0}),\beta_{0}}(\wt{\pi}')$ is isomorphic to $(\wt{\rho}\wt{\otimes}\wt{\rho}\nu^{s(\wt{\rho})})_{\wt{\omega}}$ or $(\wt{\rho}\nu^{s(\wt{\rho})}\wt{\otimes}\wt{\rho})_{\wt{\omega}}$. We let $\pi'=\mrmc(\wt{\pi}')\in\mrsqrt^{(n)}(G_{r_{0}})$. Using Proposition \ref{propmcindres}.(2), we have $$r_{(2r_{0}),\beta_{0}}(\pi')\simeq r_{(2r_{0}),\beta_{0}}(\mrmc(\wt{\pi}'))\simeq\mrmc(r_{(2r_{0}),\beta_{0}}(\wt{\pi}')),$$
		which is isomorphic to  $\mrmc(\wt{\rho})\otimes\mrmc(\wt{\rho}\nu^{s(\wt{\rho})})$ or  $\mrmc(\wt{\rho}\nu^{s(\wt{\rho})})\otimes\mrmc(\wt{\rho})$ by Proposition \ref{propmcMTPcomp}.
		Thus the cuspidal support of $\pi'$  equals the union of that of $\mrmc(\wt{\rho})$ and $\mrmc(\wt{\rho}\nu^{s(\wt{\rho})})$ (\emph{cf.} \cite{zelevinsky1980induced}). Since $\mrmc(\wt{\rho}\nu^{s(\wt{\rho})})\simeq \mrmc(\wt{\rho})\nu^{ns(\wt{\rho})}$, the union above is $$\{\rho\nu^{a},\rho\nu^{a+1},\dots,\rho\nu^{b},\rho\nu^{a+ns(\wt{\rho})},\rho\nu^{a+ns(\wt{\rho})+1},\dots,\rho\nu^{b+ns(\wt{\rho})}\}.$$ By \cite{zelevinsky1980induced}*{Sect. 9} the only possibility is $a+ns(\wt{\rho})=b+1$, meaning that $s(\wt{\rho})=m/n$.
		
	\end{proof}
	
	
	
	
	Let $m$ be a positive integer dividing $r$, let $a,b$ two real numbers such that $b-a+1=m$, and let $\rho\in\mrcusp(G_{r/m})$. Let $\beta_{m}=(r/m,\dots,r/m)$ be a composition of $r$, and we consider $G_{\beta_{m}}$ as a Levi subgroup of $G_{r}$. We let $L(\rho,[a,b])$ be the unique irreducible quotient of the parabolic induction $\rho\nu^{a}\times\rho\nu^{a+1}\times\dots\times\rho\nu^{b}:=i_{\beta_{m},(r)}(\rho\nu^{a}\otimes\rho\nu^{a+1}\otimes\dots\otimes\rho\nu^{b})$ as an essentially square integrable representation of $G_{r}$ as in \cite{zelevinsky1980induced}. We have a similar theory for $\wt{G_{r}}$
	
	\begin{lemma}[\cite{kaplan2022classification}*{Proposition 7.2}]\label{lemmaLabwtrho}
		
		Let $\wt{\rho}\in\mrcusp_{\epsilon}(\wt{G_{r}})$ and  let $\wt{\omega}$ be a compatible character of $Z(\wt{G_{r}})$ such that $(\wt{\rho}\nu^{as(\wt{\rho})}\wt{\otimes}\wt{\rho}\nu^{(a+1)s(\wt{\rho})}\wt{\otimes}\dots\wt{\otimes}\wt{\rho}\nu^{bs(\wt{\rho})})_{\wt{\omega}}$ makes sense,
		
		(1) There exists a unique irreducible subrepresentation (resp. quotient) of $$(\wt{\rho}\nu^{as(\wt{\rho})}\wt{\times}\wt{\rho}\nu^{(a+1)s(\wt{\rho})}\wt{\times}\dots\wt{\times}\wt{\rho}\nu^{bs(\wt{\rho})})_{\wt{\omega}}$$
		denoted by $Z(\wt{\rho},[a,b])_{\wt{\omega}}$, (resp. $L(\wt{\rho},[a,b])_{\wt{\omega}}$). Moreover, $L(\wt{\rho},[a,b])_{\wt{\omega}}$ is an essentially square integrable representation of $\wt{G_{r}}$;
		
		(2) For an integer $0\leq l\leq r$, we have
		$$
		r_{(r),(l,r-l)}(Z(\wt{\rho},[a,b])_{\wt{\omega}})=\begin{cases} (Z(\wt{\rho},[a,a+l'])_{\wt{\omega}_{1}}\wt{\otimes}Z(\wt{\rho},[a+l'+1,b])_{\wt{\omega}_{2}})_{\wt{\omega}}\quad &\text{if}\quad l'=ml/r\in\mbz,   \\
			0\quad & \text{if}\quad r\nmid ml.
		\end{cases}
		$$ 
		and
		$$
		r_{(r),(l,r-l)}(L(\wt{\rho},[a,b])_{\wt{\omega}})=\begin{cases} (L(\wt{\rho},[b-l'+1,b])_{\wt{\omega}_{1}'}\wt{\otimes}L(\wt{\rho},[a,b-l'])_{\wt{\omega}_{2}'})_{\wt{\omega}}\quad &\text{if}\quad l'=ml/r\in\mbz,   \\
			0\quad & \text{if}\quad r\nmid ml.
		\end{cases}
		$$ 
		where $\wt{\omega}_{1}$, $\wt{\omega}_{2}$, $\wt{\omega}_{1}'$, $\wt{\omega}_{2}'$ are compatible genuine characters.
	\end{lemma} 
	
	The following proposition classifies the image of the metaplectic correspondence of cuspidal and essentially square integrable representations.
	
	\begin{proposition}\label{propmcsqrtclassification}
		
		Let $m$ be a positive integer dividing $r$, let $\rho\in\mrcusp(G_{r/m})$ and let $\pi=L(\rho,[a,b])$ with $m=b-a+1$ satisfying $\omega_{\pi}\rest_{\mu_{n}}=1$. Let $s$ be the order of $\omega_{\rho}\rest_{\mu_{n}}$ which divides $m$.
		
		\begin{enumerate}
			\item If $\pi$ is the lift a certain $\wt{\pi}\in\mrcusp_{\epsilon}(\wt{G_{r}})$, then we have $s=m$. Conversely if $s=m$, then any $\wt{\pi}\in\mrsqrt_{\epsilon}(\wt{G_{r}})$ lifting to $\pi$ is cuspidal.
			
			\item 	Let $m'=m/s$ and let $\wt{\rho}\in\mrcusp_{\epsilon}(\wt{G_{r/m'}})$ such that $\mrmc(\wt{\rho})=L(\rho,[a,a+s-1])$.
			Then for $\omega=\omega_{\pi}$ and  a genuine character $\wt{\omega}$ of $Z(\wt{G_{r}})$ satisfying (\ref{eqomegawtomega}), we have 
			$\mrmc(L(\wt{\rho},[0,m'-1])_{\wt{\omega}})=\pi.$
			Conversely, any $\wt{\pi}\in\mrsqrt_{\wt{\omega}}(\wt{G_{r}})$ is of the form $L(\wt{\rho},[0,m'-1])_{\wt{\omega}}$ for some $m'$ dividing $r$ and some $\wt{\rho}\in\mrcusp_{\epsilon}(\wt{G_{r/m'}})$.
		\end{enumerate} 
		
	\end{proposition}
	
	\begin{proof}
		
		We first prove (1). We choose $\wt{\pi}\in\mrsqrt_{\epsilon}(\wt{G_{r}})$ such that $\mrmc(\wt{\pi})=\pi$. First we assume $\wt{\pi}$ to be cuspidal. Then for any $l'=1,\dots m-1$ and $l=l'r/m$, we have $r_{(r),(l,r-l)}(\wt{\pi})=0$. Using a similar argument to the proof of Proposition \ref{propmcindres}.(2) (more precisely, the combination of \eqref{eqmccorr} and \eqref{eqHCcharJac}) we may prove that for any $x=\mrdiag(x_{1},x_{2})\in G_{l}\times G_{r-l}$ such that $x_{1}^{n}$ and $x_{2}^{n}$ are elliptic and $\abs{\mrdet(x_{1})}^{1/l}>\abs{\mrdet(x_{2})}^{1/(r-l)}$, we have
		\begin{align*}
			0&=\bigg(\sum_{\{t_{1}\mid t_{1}^{n}=x_{1}^{n}\}}\theta_{L(\rho,[b-l'+1,b])}(t_{1})\cdot\epsilon(x_{1}^{*}/t_{1}^{*})\bigg)\cdot\sum_{\{t_{2}\mid t_{2}^{n}=x_{2}^{n}\}}\theta_{L(\rho,[a,b-l'])}(t_{2})\cdot\epsilon(x_{2}^{*}/t_{2}^{*})\\
			&=\theta_{L(\rho,[b-l'+1,b])}(x_{1})\cdot\theta_{L(\rho,[a,b-l'])}(x_{2})\cdot\bigg(\sum_{\zeta_{1}\in\mu_{n}}\omega_{\rho}(\zeta_{1})^{l'}\bigg)\cdot\bigg(\sum_{\zeta_{2}\in\mu_{n}}\omega_{\rho}(\zeta_{2})^{m-l'}\bigg),
		\end{align*}
		where for the second equation we use the fact that $x_{i}$ is elliptic, and thus $x_{i}=t_{i}\zeta_{i}$ for a certain $\zeta_{i}\in\mu_{n}$ and $x_{i}^{*}=t_{i}^{*}$.
		Since $l'$ and $x$ are arbitrary,  $\omega_{\rho}\rest_{\mu_{n}}$ must be of order $m$. Conversely if $\omega_{\rho}\rest_{\mu_{n}}$ is of order $m$, then for any $1\leq l\leq r-1$ the representation $r_{(r),(l,r-l)}(L(\rho,[a,b]))$ is not in $\mrsqrt_{\omega}^{(n)}(G_{l}\times G_{r-l})$, thus using Proposition \ref{propmcindres}.(2) $r_{(r),(l,r-l)}(\wt{\pi})=0$. So $\wt{\pi}$ is cuspidal. 
		
		Now we prove (2). We have that $\mrmc(L(\wt{\rho},[0,m'-1])_{\wt{\omega}})$ is in $\mrsqrt(G_{r})$ and by Proposition \ref{propvaluesrho}, $s(\wt{\rho})=s/n$. Using Proposition \ref{propmclift}, Proposition \ref{propmcindres}.(2) and Lemma \ref{lemmaLabwtrho}.(2), for the composition $\beta_{m'}=(r/m',\cdots,r/m')$ of $r$, we have
		\begin{align*}
			&r_{(r),\beta_{m'}}(\mrmc(L(\wt{\rho},[0,m'-1])_{\wt{\omega}}))=\mrmc(r_{(r),\beta_{m'}}(L(\wt{\rho},[0,m'-1])_{\wt{\omega}}))\\
			=&\mrmc((\wt{\rho}\nu^{s(\wt{\rho})(m'-1)}\wt{\otimes}\dots\wt{\otimes}\wt{\rho}\nu^{s(\wt{\rho})}\wt{\otimes}\wt{\rho})_{\wt{\omega}})\\
			=&L(\rho,[a+(m'-1)s,b])\otimes\dots \otimes L(\rho,[a+s,a+2s-1])\otimes L(\rho,[a,a+s-1])  .
		\end{align*}
		Thus the cuspidal support of $\mrmc(L(\wt{\rho},[0,m'-1])_{\wt{\omega}})$ is $\{\rho\nu^{a},\rho\nu^{a+1},\dots,\rho\nu^{b}\}$,  implying that $\pi=\mrmc(L(\wt{\rho},[0,m'-1])_{\wt{\omega}})$. Finally any $\wt{\pi}\in\mrsqrt_{\wt{\omega}}(\wt{G_{r}})$ lifts to a certain $\pi=L(\rho,[a,b])$ as above. By Proposition \ref{propmclift} we must have $\wt{\pi}=L(\wt{\rho},[0,m'-1])_{\wt{\omega}}$ for $m'=m/s$ and $\wt{\rho}$ satisfying $\mrmc(\wt{\rho})=L(\rho,[a,a+s-1])$.
		
	\end{proof}
	
	\subsection{Classification of essentially tempered representations}
	
	We further study essentially tempered representations.
	
	\begin{proposition}\label{proptempclasify}
		
		\begin{enumerate}
			\item	Let $\wt{\pi}_{i}$ be a genuine irreducible square integrable representation of $\wt{G_{r_{i}}}$ for each $i$ with $r_{1}+\dots+r_{k}=r$ and let $\wt{\omega}$ be a genuine character of $Z(\wt{G_{r}})$ satisfying \eqref{eqomegacompatible}, then $(\wt{\pi}_{1}\wt{\times}\dots\wt{\times}\wt{\pi}_{k})_{\wt{\omega}}$ is an irreducible and tempered representation of $\wt{G_{r}}$.
			
			\item Conversely every genuine irreducible tempered representation $\wt{\pi}$ of $\wt{G_{r}}$ is of the form $(\wt{\pi}_{1}\wt{\times}\dots\wt{\times}\wt{\pi}_{k})_{\wt{\omega}}$ as above. Such $\wt{\pi}$ correspond in bijection with $\wt{G_{r}}$-conjugacy classes of the metaplectic tensor product $(\wt{\pi}_{1}\wt{\otimes}\dots\wt{\otimes}\wt{\pi}_{k})_{\wt{\omega}}$ with $r_{i}$, $\wt{\pi}_{i}$, $\wt{\omega}$ as in (1).

		\end{enumerate}

	\end{proposition}
	
	\begin{proof}
		
		This proposition was first noted in \cite{flicker1986metaplectic}*{ Proposition 27, p98}, with an essential usage of the metaplectic correspondence (\emph{cf.} \cite{flicker1986metaplectic}*{Lemma 27, p98}) as well as some non-trivial results in the book of Silberger that are not explicitly written down for a covering group. Here we give a more elementary proof instead.
		
		Using Proposition \ref{propmcsqrtclassification}, for each $i$ we may find a certain positive integer $m_{i}$ dividing $r_{i}$, a genuine unitary cuspidal representation $\wt{\rho}_{i}$ of $\wt{G_{r_{i}/m_{i}}}$ and a genuine character $\wt{\omega}_{i}$ of $Z(\wt{G_{r_{i}}})$ such that
		$\wt{\pi}_{i}=L(\wt{\rho}_{i},[(-m_{i}+1)/2,(m_{i}-1)/2])_{\wt{\omega}_{i}}$. 
		
		\begin{lemma}
			
			For a compatible genuine character  $\wt{\omega}$ of $Z(\wt{G_{r}})$, the representation 
			$$(L(\wt{\rho}_{1},[(-m_{1}+1)/2,(m_{1}-1)/2])_{\wt{\omega}_{1}}\wt{\times}\dots\wt{\times} L(\wt{\rho}_{k},[(-m_{k}+1)/2,(m_{k}-1)/2])_{\wt{\omega}_{k}})_{\wt{\omega}}$$ 
			is irreducible.
			
		\end{lemma}
		
		\begin{proof}
			
			Write $\Delta_{i}$ for the segment $\{\wt{\rho}_{i}\nu^{(-m_{i}+1)s(\wt{\rho}_{i})/2},\wt{\rho}_{i}\nu^{(-m_{i}+3)s(\wt{\rho}_{i})/2},\dots,\wt{\rho}_{i}\nu^{(m_{i}-1)s(\wt{\rho}_{i})/2}\}$, then the lemma follows from the fact that $\Delta_{i}$ are pairwise weakly unlinked. See \cite{kaplan2022classification}*{Theorem 7.5}.
			
		\end{proof}
		
		Using this lemma and the fact that the parabolic induction of a square integrable representation is (not necessarily irreducible) tempered (\emph{cf.} \cite{waldspurger2003formule}*{Lemma III.2.3}), $(\wt{\pi}_{1}\wt{\times}\dots\wt{\times}\wt{\pi}_{k})_{\wt{\omega}}$ is tempered. The second statement follows readily from the following lemma.
		
		\begin{lemma}\cite{waldspurger2003formule}*{Proposition III.4.1} Any irreducible tempered representation $\wt{\pi}$ is a direct summand of a certain parabolic induction $i_{\beta,(r)}(\wt{\rho})$, where $\wt{\rho}$ is an irreducible square integrable representation of a Levi subgroup $\wt{G_{\beta}}$ of $\wt{G_{r}}$. Such a pair $(\wt{G_{\beta}},\wt{\rho})$ is unique up to $\wt{G_{r}}$-conjugacy.
			
		\end{lemma}
		
		We remark that although the results we cited above are not for covering groups, the same argument indeed works for a general finite central extension of a $p$-adic reductive group. Note that the argument is based on the definition of square integrable and tempered representations via exponents, which is known in general \cite{ban2013langlands}*{Theorem 3.4 and Theorem 3.5}. So the original proof can also be generalized without difficulty.
		

		
	\end{proof}
	
	Finally we study the metaplectic correspondence for essentially tempered representations. Let  $r'=r_{1}'+\dots+r_{l}'$ and  $\pi'=\rho_{1}\times\dots\times\rho_{l}\in\mrtemp(G_{r'})$ with $\rho_{i}\in\mrsqrt(G_{r_{i}'})$. The representation $\pi'$ is called \emph{metic}\footnote{It is an abbreviation of metaplectic, a notation introduced by Flicker-Kazhdan.} if each central character $\omega_{\rho_{i}}$ restricted to $\mu_{n}$ is trivial. In general, a representation $\pi=\pi_{1}\otimes\dots\otimes\pi_{k}\in\mrtemp(G_{\beta})$ with each $\pi_{i}\in\mrtemp(G_{r_{i}})$ is called \emph{metic} if each $\pi_{i}$ is metic, where $\beta=(r_{1},\dots,r_{k})$ is a composition of $r$ as before. We denote by $\mrtemp_{\omega}(G_{\beta})$ the set of equivalence classes of irreducible essentially tempered representations whose central character restricted to $Z(G_{r})$ is $\omega$, and by $\mrtemp^{(n)}_{\omega}(G_{\beta})$ its subset consisting of metic ones.
	
	\begin{proposition}\label{propmclifttemp}
		
		For any $\pi\in\mrtemp^{(n)}_{\omega}(G_{\beta})$, there exists a unique $\wt{\pi}\in\mrtemp_{\wt{\omega}}(\wt{G_{\beta}})$ that lifts to $\pi$, and conversely any such $\wt{\pi}$ lifts to some $\pi$.
		
	\end{proposition}
	
	\begin{proof}
		
		We first consider $\beta=(r)$ and $G_{\beta}=G_{r}$. We write $\wt{\pi}=(\wt{\rho}_{1}\wt{\times}\dots\wt{\times}\wt{\rho}_{l})_{\wt{\omega}}$ with $\wt{\rho}_{i}$ essentially square integrable. We let $\rho_{i}=\mrmc(\wt{\rho}_{i})$ for each $i$. By Proposition \ref{propmcMTPcomp} and Proposition \ref{propmcindres}.(1), $\wt{\pi}$ lifts to $\rho_{1}\times\dots\times\rho_{l}$ that is essentially tempered and metic. Conversely let $\pi=\rho_{1}\times\dots\times\rho_{l}$ with each $\rho_{i}$ essentially square integrable, such that the restriction of $\omega_{\rho_{i}}$ to $\mu_{n}$ is trivial. Using Proposition \ref{propmclift} we choose $\wt{\rho}_{i}$ to be an essentially square integrable representation such that $\mrmc(\wt{\rho}_{i})=\rho_{i}$, then $(\wt{\rho}_{1}\wt{\times}\dots\wt{\times}\wt{\rho}_{l})_{\wt{\omega}}$ lifts to $\pi$. For general $\beta$ using Proposition \ref{propmcMTPcomp} and the above case, similarly we may prove that every $\wt{\pi}\in\mrtemp^{(n)}_{\omega}(G_{\beta})$ lifts to a $\pi\in\mrtemp_{\omega}(G_{\beta})$, and conversely every $\pi\in\mrtemp_{\omega}(G_{\beta})$ is a lift of a certain $\wt{\pi}\in\mrtemp^{(n)}_{\omega}(G_{\beta})$. Finally such $\widetilde{\pi}$ lifting to $\pi$ is unique by Remark \ref{remTrpitildesupp}.
		
	\end{proof}
	
	However in the above proposition, it is unclear to the author if  $\pi\in\mrtemp^{(n)}_{\omega}(G_{\beta})$, as a lift of a certain $\wt{\pi}\in\mrtemp_{\wt{\omega}}(\wt{G_{\beta}})$, is unique or not. So until now we cannot claim the existence of a map
	$\mrmc:\mrtemp_{\wt{\omega}}(\wt{G_{\beta}})\rightarrow\mrtemp_{\omega}^{(n)}(G_{\beta})$, $\wt{\pi}\mapsto\pi.$

	\section{Calculation of Whittaker dimension}
	
	In this section, we consider one of the main important applications of metaplectic correspondence: calculating the Whittaker dimension of a representation $\wt{\pi}$ in $\mrirr_{\epsilon}(\wt{G_{r}})$.
	
	\subsection{A general introduction of the Harish-Chandra germ expansion}\label{subsectionHCgerm}
	
	In this subsection, let $G$ be a connected reductive group over $F$, let $\mfg$ be the Lie algebra of $G$. We won't recall concrete definition of some notation in this part, but indeed the only case we are interested in is $G$ being a Levi subgroup $G_{\beta}$ of $G_{r}$. So the readers may just imagine what happens in general and refer to the reference we provide.
	
	Let $\pi$ be an irreducible representation of $G$. First we briefly recall the definition of the Harish-Chandra germ function (\emph{cf.} \cite{harishchandra1999admissible}). Let $\theta_{\pi}$ be the Harish-Chandra character of $\pi$, which has the following germ expansion
	$$\theta_{\pi}(x\mrexp(X))=\sum_{\mco\in\mrnil(\mfg_{x})}c_{\pi,\mco}(x)\hat{\jmath}(X,\mco),$$
	where 
	\begin{itemize}
		\item $x$ is a semi-simple element of $G$;
		\item $X$ is a semi-simple regular element in a small neighborhood of $0$ in $\mfg$;
		\item $\mrexp:\mfg\rightarrow G$ denotes the exponential map;
		\item $\mfg_{x}$ denotes the Lie algebra of centralizer of $x$ in $G$;
		\item $\mrnil(\mfg_{x})$ denotes the set of nilpotent orbits of $\mfg_{x}$;
		\item $\hat{\jmath}(\cdot,\mco)$ denotes the Fourier transform of the orbital integral along $\mco$, which is a smooth function on semi-simple regular elements of $\mfg$;
		\item $c_{\pi,\mco}(x)$ is a complex number depending on $\pi$, $\mco$ and $x$.
	\end{itemize} 
	We write $\mrnil_{\mrreg}(\mfg_{x})$ for the set of regular nilpotent orbits of $\mfg_{x}$, 
	and we define $$c_{\pi}(x):=\car{\mrnil_{\mrreg}(\mfg_{x})}^{-1}\sum_{\mco_{0}\in\mrnil_{\mrreg}(\mfg_{x})}c_{\pi,\mco_{0}}(x)$$ as a complex function on the set of semi-simple elements in $G$, called the \emph{Harish-Chandra germ function} of $\pi$. The following formula can be used to calculate this germ function (\cite{beuzart2020local}*{ Proposition 4.5.1}):
	\begin{equation}\label{eqcpilimtheta}
		\Delta^{G}(x)c_{\pi}(x)=
		\begin{cases}
			\car{W(G_{x},T_{x})}^{-1}\lim_{x'\rightarrow x}\Delta^{G}(x')\theta_{\pi}(x'),\quad&\text{if}\ G_{x}\ \text{is quasi-split},\\
			0, \quad&\text{otherwise},
		\end{cases}
	\end{equation}
	where 
	\begin{itemize}
		\item $D^{G}(x)$ denotes the Weyl discriminant of $G$ and $\Delta^{G}(x)=D^{G}(x)^{1/2}$;
		\item $G_{x}$ denotes the neutral connected component of the centralizer of $x$ in $G$;
		\item when $G_{x}$ is quasi-split, $T_{x}$ denotes the maximal split torus contained in a Borel subgroup of $G_{x}$;
		\item $W(G_{x},T_{x})$ denotes the corresponding Weyl group;
		\item $x'$ in the limit are semi-simple regular elements in $T_{x}$. 
	\end{itemize}  
	
	Similarly we consider the corresponding theory for a covering group. Let $\wt{G}$ be a central extension of $G$ by $\mu_{n}$ and let $\wt{\pi}$ be an irreducible representation of $\wt{G}$. Let $x^{*}\in\wt{G}$ such that $\bs{p}(x^{*})$ is semi-simple. To simplify our discussion, we also assume $x^{*}$ to be \emph{good}, meaning that for any $y\in G$ commuting with $\bs{p}(x^{*})$, we have that $x^{*}$  and $\bs{s}(y)$ commute. Indeed, when $x^{*}$ is semi-simple and good, the character $[x^{*},\cdot]:\wt{G_{\bs{p}(x^{*})}}\rightarrow\mu_{n}$ is trivial, where $[\cdot,\cdot]$ denotes the commutator in $\wt{G}$. Thus in this special case the results in \cite{li2012formule}*{Section 4} are largely simplified. 
	
	For the Harish-Chandra character $\theta_{\wt{\pi}}$ we still have the following germ expansion (\cite{li2012formule}*{Th\'eor\`eme 4.3.2}):
	$$\theta_{\wt{\pi}}(x^{*}\mrexp(X))=\sum_{\mco\in\mrnil(\mfg_{\bs{p}(x^{*})})}c_{\wt{\pi},\mco}(x^{*})\hat{\jmath}(X,\mco),$$
	where 
	\begin{itemize}
		\item $X$ is a semi-simple regular element in a small neighborhood of $0$ in $\mfg$;
		\item $\mrexp(X)$ is realized as an element in $\wt{G}$ via the splitting $\bs{s}$;
		\item $c_{\wt{\pi},\mco}(x^{*})$ is a complex number depending on $\wt{\pi}$, $\mco$ and $x^{*}$.
	\end{itemize}
	We similarly define
	$$c_{\wt{\pi}}(x^{*}):=\car{\mrnil_{\mrreg}(\mfg_{\bs{p}(x^{*})})}^{-1}\sum_{\mco_{0}\in\mrnil_{\mrreg}(\mfg_{\bs{p}(x^{*})})}c_{\wt{\pi},\mco_{0}}(x^{*})$$ as a complex function on the set of good semi-simple elements in $G$, called the \emph{Harish-Chandra germ function} of $\wt{\pi}$. An identical argument of \cite{beuzart2020local}*{ Proposition 4.5.1} gives us 
	\begin{equation}\label{eqcpilimthetacover}
		\Delta^{G}(\bs{p}(x^{*}))c_{\wt{\pi}}(x^{*})=
		\begin{cases}
			\car{W(G_{\bs{p}(x^{*})},T_{\bs{p}(x^{*})})}^{-1}\lim_{x'^{*}\rightarrow x^{*}}\Delta^{G}(\bs{p}(x'^{*}))\theta_{\wt{\pi}}(x'^{*}),\quad&\text{if}\ G_{\bs{p}(x^{*})}\ \text{is quasi-split},\\
			0, \quad&\text{otherwise,}
		\end{cases}
	\end{equation}
	where $x'^{*}$ in the limit are good semi-simple regular elements in $\wt{T_{\bs{p}(x^{*})}}$.

	\subsection{Whittaker space and dimension}
	
	We come back to our study of Kazhdan-Patterson covering groups. Let $\beta=(r_{1},\dots,r_{k})$ be a composition of $r$, and we write $G=G_{\beta}$ to simplify our notation. Let $N$ be the unipotent subgroup of a Borel subgroup of $G$. Let $\psi$ be a generic character of $N$, which means that for any simple root $\alpha$ related to $N$ and the corresponding subgroup $N_{\alpha}$ of $N$, the restriction $\psi\rest_{N_{\alpha}}$ is non-trivial. It is easy to verify that such kind of pairs $(N,\psi)$ form a single $G$-conjugacy class. 
	
	For a finite length genuine representation $\wt{\pi}$ of $\wt{G}$, we define 
	$$\mrwh(\wt{\pi})=\mrhom_{N}(\wt{\pi},\psi)$$
	the \emph{Whittaker space} of $\wt{\pi}$, where as before we regard $N$ as a subgroup of $\wt{G}$. As we explained, this vector space essentially does not depend on the choice of the pair $(N,\psi)$.
	
	We note that $\mrwh$, as a functor from $\mrrep_{\epsilon}(\wt{G})$ to the category of complex vector spaces, is exact. Since when $\beta=(r)$, using the argument in \cite{kazhdan1984metaplectic}*{Theorem I.5.3} the functor $\mrwh$ is realized by taking the $r$-th Bernstein-Zelevinsky derivative, which is an exact functor. In general, for each $i$ we may similarly define $r_{i}$-th ``partial Bernstein-Zelevinsky functor" with respect to the $i$-th block in $G$, which is also exact. The functor $\mrwh$ is realized by taking the $r_{i}$-th partial Bernstein-Zelevinsky functor for each $i$.
	
	We define $d_{\wt{\pi}}:=\mrdim_{\mbc}(\mrwh(\wt{\pi}))$, which is finite because of the following proposition.
	
	\begin{proposition}[\cite{kazhdan1984metaplectic}*{Theorem I.5.3} or \cite{patel2015theorem}*{Theorem 2}]\label{propwhitdim}
		
		For $\wt{\pi}$ irreducible we have $d_{\wt{\pi}}=c_{\wt{\pi}}(I_{r})$. 
		
	\end{proposition}
	
	Let $\wt{\pi}_{i}\in\mrirr_{\epsilon}(\wt{G_{r_{i}}})$ for each $i=1,\dots,k$ and let $\wt{\omega}$ be a genuine character of $Z(\wt{G_{r}})$ satisfying \eqref{eqomegacompatible}.
	
	\begin{proposition} \label{propwhitdimBZprod}
		We have
		$ d_{(\wt{\pi}_{1}\wt{\times}\dots\wt{\times}\wt{\pi}_{k})_{\wt{\omega}}}=d_{(\wt{\pi}_{1}\wt{\otimes}\dots\wt{\otimes}\wt{\pi}_{k})_{\wt{\omega}}}=\dF{d_{r}}^{-1}\cdot\prod_{i=1}^{k}\dF{d_{r_{i}}}\cdot d_{\wt{\pi}_{i}}$.
		
	\end{proposition}
	
	\begin{proof}
		
		The first equation follows from \cite{banks1998heredity}*{Theorem}. The second equation follows from Corollary \ref{corpirest} and the fact that the unipotent radical $N$ is contained in $G_{\beta,\beta}^{(n)}$.
		
	\end{proof}
	
	Here, $d_{r}=\mrgcd(n,2rc-r+1)$ is defined as in \S \ref{subsectionMTP}.
	
	Let $\wt{\omega}$ and $\omega$ be characters satisfying \eqref{eqomegawtomega}. Let $\wt{\pi}\in\mrsqrt_{\wt{\omega}}(\wt{G})$ and let $\pi\in\mrsqrt_{\omega}(G)$ be the metaplectic lift of $\wt{\pi}$. We study the relation between the Harish-Chandra germ functions $c_{\wt{\pi}}$ and $c_{\pi}$.
	
	\begin{proposition}\label{propcwtpicpi}
		
		Let $x$ be a semi-simple element in $G$ such that $x^{n}$ is semi-simple and the centralizers $G_{x}$ and $G_{x^{n}}$ are equal. Let $T$ be a maximal split torus contained in $G_{x}$ that contains $x$. Then
		$$\Delta^{G}(x^{n})\cdot\car{W(G_{x},T)}\cdot c_{\wt{\pi}}(x^{*})=\frac{1}{\dF{d_{r}}\abs{n}^{r/2}}\sum_{\{t\in T\mid t^{n}=x^{n}\}}\Delta^{G}(t)\cdot\car{W(G_{t},T)}\cdot c_{\pi}(t)\cdot\epsilon(x^{*}/t^{*}).$$
		
	\end{proposition}
	
	\begin{proof}
		
		We choose $x'$ to be an element in $T$ sufficiently close to $x$, such that $x'^{n}$ is semi-simple regular. Using \eqref{eqmccorr} we have
		$$\Delta^{G}(x'^{n})\cdot \theta_{\wt{\pi}}(x'^{*})=\frac{1}{\dF{d_{r}}\abs{n}^{r/2}}\sum_{\{t'\in T\mid t'^{n}=x'^{n}\}}\Delta^{G}(t')\cdot \theta_{\pi}(t')\cdot\epsilon(x'^{*}/t'^{*}).$$
		Taking the limit $x'\rightarrow x$ and using \eqref{eqcpilimtheta} and \eqref{eqcpilimthetacover}, we finish the proof.
		
	\end{proof}
	
	\begin{remark}
		
		In Proposition \ref{propcwtpicpi} the assumption $G_{x}=G_{x^{n}}$ is important, which guarantees that $T$ is a maximal split torus of each $G_{t}$ and $G_{x^{n}}$. We don't know what happens if this condition is not satisfied.
		
	\end{remark}
	
	Taking $x=I_{r}$ and using Proposition \ref{propwhitdim} and Proposition \ref{propcwtpicpi}, we have the following important corollary.
	
	\begin{corollary}
		We have
		\begin{equation}\label{eqwhitdim}
			d_{\wt{\pi}}=\frac{1}{\dF{d_{r}}\abs{n}^{r/2}r!}\sum_{t}\Delta^{G}(t)\cdot\car{W(G_{t},T)}\cdot c_{\pi}(t),	
		\end{equation}
		where $t$ in the sum ranges over $\mrdiag(\zeta_{1},\dots,\zeta_{r})$ with $\zeta_{1},\dots,\zeta_{r}\in\mu_{n}$, and $T$ denotes the torus of diagonal matrices.
		
	\end{corollary}

	\subsection{The main Whittaker dimensional formula for essentially square integrable representations}
	
	In this subsection we focus on the case where $G=G_{r}$. We let $T$ be the diagonal torus of $G$. In general we let $T_{r'}$ be the diagonal torus of $G_{r'}$ for any $r'\geq 1$.
	
	First we propose the following conjecture related to a so-called ``Zelevinsky standard module". Assume $r=r_{0}m$ for positive integers $r_{0}$ and $m$. For $\rho\in\mrcusp(G_{r_{0}})$ and two real numbers $a,b$ such that $b-a+1=m$, we define
	$Z(\rho,[a,b])$ to be the unique irreducible subrepresentation of 
	$$\rho\nu^{a}\times\rho\nu^{a+1}\times\dots\times\rho\nu^{b}.$$
	
	\begin{conjecture}\label{conjmain}
		
		For $x=\mrdiag(\zeta_{1},\dots,\zeta_{r})$ with $\zeta_{i}\in\mu_{n}$, we have
		$$c_{Z(\rho,[a,b])}(x)=\begin{cases} \prod_{i=1}^{m}\omega_{\rho}(\zeta_{i}')\quad&\text{if}\ x\ \text{is}\ G\text{-conjugate to}\ \mrdiag(\zeta_{1}'I_{r_{0}},\dots, \zeta_{m}'I_{r_{0}}),\\ & \zeta_{i}'\in \mu_{n}\ \text{pairwise different}\ \text{for}\ i=1,\dots,m; \\
			0\quad &\text{otherwise}.
		\end{cases}$$
		
	\end{conjecture}
	
	This conjecture helps us to calculate the Whittaker dimension of a genuine representation of $\wt{G}$. We state the following main theorem.
	
	\begin{theorem}\label{thmwhitcal}
		
		Let $\wt{\pi}\in\mrsqrt_{\epsilon}(\wt{G})$. Let $m$ be the positive integer dividing $r$ and let $\rho\in\mrcusp(G_{r/m})$, such that $\mrmc(\wt{\pi})=L(\rho,[0,m-1])$ (\emph{cf.} Proposition \ref{propmcsqrtclassification}). Let $r_{0}=r/m$ and let $s$ be the order of $\omega_{\rho}\rest_{\mu_{n}}$ that divides $m$. When $\mrgcd(n,p)=1$, we have
		$$d_{\wt{\pi}}=\frac{1}{d_{r}}\cdot\binom{m/s+n/s-1}{m/s}.$$ 
		
	\end{theorem}
	
	\begin{proof} 
		
		\textbf{We prove Theorem \ref{thmwhitcal} on assuming Conjecture \ref{conjmain}.} First we prove a general lemma concerning the Harish-Chandra germ function. For a finite length representation $\pi$ of $G$ having irreducible subquotients $\pi_{1},\dots,\pi_{l}$ (counting the multiplicity), we define $\theta_{\pi}:=\sum_{i=1}^{l}\theta_{\pi_{i}}$ and $c_{\pi}:=\sum_{i=1}^{l}c_{\pi_{i}}$.
		
		\begin{lemma}\label{lemmacBZproduct}
			
			Let $\beta=(r_{1},\dots,r_{k})$ be a composition of $r$, let $\pi_{i}\in\mrirr(G_{r_{i}})$, let $T$ be the diagonal torus of $G$ and $G_{\beta}$, and let $x=\mrdiag(\zeta_{1},\dots,\zeta_{r})\in T$ such that $\zeta_{i}\in F^{\times}$. Then $$\car{W(G_{x},T)}\cdot c_{\pi_{1}\times\dots\times\pi_{k}}(x)=\frac{1}{r_{1}!\dots r_{k}!}\sum_{w\in W(G,T)}\prod_{i=1}^{k}\car{W((G_{r_{i}})_{x_{w,i}},T_{r_{i}})}\cdot c_{\pi_{i}}(x_{w,i}),$$
			where for $x\in T$ and $w\in W(G,T)$ we write $w^{-1}xw=\mrdiag(x_{w,1},\dots, x_{w,k})$ with $x_{w,i}\in G_{r_{i}}$, and $(G_{r_{i}})_{x_{w,i}}$ denotes the centralizer of $x_{w,i}$ in $G_{r_{i}}$.
			
		\end{lemma}
		
		\begin{proof}
			
			We consider $x'=\mrdiag(\zeta_{1}',\dots,\zeta_{r}')$ with $\zeta_{i}'\in F^{\times}$ pairwise different for $i=1,\dots,r$. We denote by $W(G,T)$ and $W(G_{\beta},T)$ the corresponding Weyl groups.
			Using \cite{van1972computation}*{Theorem 3} we have
			\begin{align*}
				\Delta^{G_{r}}(x')\theta_{\pi}(x')&=\sum_{\{g\in G_{r}/P_{(r),\beta}\mid g^{-1}x'g\in G_{\beta}\}}\Delta^{G_{\beta}}(g^{-1}x'g)\theta_{\pi_{1}\otimes\dots\otimes\pi_{k}}(g^{-1}x'g)\\
				&=\sum_{w\in W(G, T)/W(G_{\beta},T)}\Delta^{G_{\beta}}(w^{-1}x'w)\theta_{\pi_{1}\otimes\dots\otimes\pi_{k}}(w^{-1}x'w)\\
				&=\car{W(G_{\beta},T)}^{-1}\sum_{w\in W(G,T)}\Delta^{G_{\beta}}(w^{-1}x'w)\theta_{\pi_{1}\otimes\dots\otimes\pi_{k}}(w^{-1}x'w)\\
				&=\frac{1}{r_{1}!\dots r_{k}!}\sum_{w\in W(G,T)}\prod_{i=1}^{k}\Delta^{G_{r_{i}}}(x'_{w,i})\theta_{\pi_{i}}(x'_{w,i})
			\end{align*}
			Taking the limit $x'\rightarrow x$ and using \eqref{eqcpilimtheta} we finish the proof.
			
		\end{proof}

			Now we focus on the proof of Theorem \ref{thmwhitcal}. We write $\pi=\mrmc(\wt{\pi})=L(\rho,[0,m-1])$. Using Tadic's determinantal formula \cite{lapid2014determinantal}*{Theorem 1}, we have
			$$\pi=\sum_{m_{1}+\dots+m_{l}=m}(-1)^{m-l}Z(\rho,[0,t_{1}-1])\times Z(\rho,[t_{1},t_{2}-1])\times\dots\times Z(\rho,[t_{l-1},t_{l}-1]),$$
			where $m_{i}$ in the sum are positive integers, and we define $t_{0}=0$ and $t_{i}=m_{1}+\dots+m_{i}$ for each $i=1,2,\dots,l$, and the equation is taken in the Grothendieck group of finite length representations of $G$. Taking the Harish-Chandra germ function and using Lemma \ref{lemmacBZproduct}, we get
			\begin{align*}&\car{W(G_{x},T)}\cdot c_{\pi}(x)=\car{W(G_{x},T)}\sum_{m_{1}+\dots+m_{l}=m}(-1)^{m-l}c_{(\times_{i=1}^{l}Z(\rho,[t_{i-1},t_{i}-1]))}(x)\\
				&=\sum_{m_{1}+\dots+m_{l}=m}\frac{(-1)^{m-l}}{\prod_{i=1}^{l}(r_{0}m_{i})!}\sum_{w\in W(G,T)}\prod_{i=1}^{l}\car{W((G_{r_{0}m_{i}})_{x_{w,i}},T_{r_{0}m_{i}})}\cdot c_{Z(\rho,[t_{i-1},t_{i}-1])}(x_{w,i}).
			\end{align*}
			Using \eqref{eqwhitdim}, we have
			\begin{align*}
				d_{\wt{\pi}}&=\frac{1}{\dF{d_{r}}\abs{n}^{r/2}r!}\sum_{x}\Delta^{G}(x)\cdot\car{W(G_{x},T)}\cdot c_{\pi}(x)\\
				&=\frac{1}{\dF{d_{r}}\abs{n}^{r/2}r!}\sum_{m_{1}+\dots+m_{l}=m}\frac{(-1)^{m-l}}{\prod_{i=1}^{l}(r_{0}m_{i})!}\sum_{w\in W(G,T)}\sum_{x}f(m_{1},\dots,m_{l},w,x)
			\end{align*}
			where $x$ ranges over $\mrdiag(\zeta_{1},\dots,\zeta_{r})$ with $\zeta_{i}\in\mu_{n}$ for each $i$, and $$f(m_{1},\dots,m_{l},w,x):=\Delta^{G}(x)\cdot\prod_{i=1}^{l}\car{W((G_{r_{0}m_{i}})_{x_{w,i}},T_{r_{0}m_{i}})}\cdot c_{Z(\rho,[t_{i-1},t_{i}-1])}(x_{w,i}).$$
			It is clear that
			$$\sum_{w\in W(G,T)}\sum_{x}f(m_{1},\dots,m_{l},w,x)=\car{W(G,T)}\sum_{x}f(m_{1},\dots,m_{l},1,x),$$
			so we further have
			\begin{equation}\label{eqdtildepi}
				d_{\wt{\pi}}=\frac{1}{\dF{d_{r}}\abs{n}^{r/2}}\sum_{m_{1}+\dots+m_{l}=m}\frac{(-1)^{m-l}}{\prod_{i=1}^{l}(r_{0}m_{i})!}\sum_{x}f(m_{1},\dots,m_{l},1,x).
			\end{equation}
			For the identity element $1\in W(G,T)$, we write $x_{i}=x_{1,i}$ for short for each $i=1,\cdots,l$. Using Conjecture \ref{conjmain}, 
			$$f(m_{1},\dots,m_{l},1,x)=\Delta^{G}(x)\prod_{i=1}^{l}\car{W((G_{r_{0}m_{i}})_{x_{i}},T_{r_{0}m_{i}})}\cdot c_{Z(\rho,[t_{i-1},t_{i}-1])}(x_{i})$$
			is 0 except the following condition holds: For every $i=1,\dots,l$, the matrix $x_{i}$ is $G_{r_{0}m_{i}}$-conjugate to a matrix of the form $\mrdiag(\zeta_{i1}I_{r_{0}},\dots,\zeta_{im_{i}}I_{r_{0}})$ with $\zeta_{ij}\in\mu_{n}$ pairwise different for $j=1,\cdots,m_{i}$. In this case we have $\car{W((G_{r_{0}m_{i}})_{x_{i}},T_{r_{0}m_{i}})}=(r_{0}!)^{m_{i}}$ for each $i$ and $$f(m_{1},\dots,m_{l},1,x)=\Delta^{G}(x) \prod_{i=1}^{l}(r_{0}!)^{m_{i}}\prod_{i=1}^{m_{i}}\omega_{\rho}(\zeta_{ij}).$$ 
			
			\textbf{From now on we assume $\mrgcd(n,p)=1$.} The most important consequence of this assumption is that, $\Delta^{G}(x)=1$ for any $x$ of the form $\mrdiag(\zeta_{1},\dots,\zeta_{n})$ with $\zeta_{i}\in\mu_{n}$, which seems to be crucial in the calculation below. Moreover, we also have $\mu_{n}\subset\mbf_{q}^{\times}$, $\dF{d_{r}}=d_{r}$ and $\abs{n}=1$.
			
			Then we have
			\begin{align*}
				d_{\wt{\pi}}&=\frac{1}{d_{r}}\sum_{m_{1}+\dots+m_{l}=m}\bigg(\prod_{i=1}^{l}\frac{(-1)^{m_{i}-1}(r_{0}!)^{m_{i}}}{(r_{0}m_{i})!}\sum_{x_{i}}\prod_{j=1}^{m_{i}}\omega_{\rho}(\zeta_{ij})\bigg)\\
				&=\frac{1}{d_{r}}\sum_{m_{1}+\dots+m_{l}=m}\bigg(\prod_{i=1}^{l}\frac{(-1)^{m_{i}-1}}{m_{i}!}\sum_{\substack{\zeta_{i1},\dots,\zeta_{im_{i}}\in\mu_{n}\\ \text{pairwise different}} }\prod_{j=1}^{m_{i}}\omega_{\rho}(\zeta_{ij})\bigg)
			\end{align*}
			where for each $i$, the corresponding $x_{i}$ in the sum ranges over the matrices in $T_{r_{0}m_{i}}$ that are $G_{r_{0}m_{i}}$-conjugate to $\mrdiag(\zeta_{i1}I_{r_{0}},\dots,\zeta_{im_{i}}I_{r_{0}})$ with $\zeta_{ij}\in\mu_{n}$ pairwise different. 
			
			\begin{lemma}
				For each $i$ we have
				$$\frac{1}{m_{i}!}\sum_{\substack{\zeta_{i1},\dots,\zeta_{im_{i}}\in\mu_{n}\\ \text{pairwise different}} }\prod_{j=1}^{m_{i}}\omega_{\rho}(\zeta_{ij})=\begin{cases}
					(-1)^{(s+1)m_{i}/s}\cdot\binom{n/s}{m_{i}/s}\quad&\text{if}\ s\ \text{divides}\ m_{i};\\
					0\quad&\text{otherwise}.
				\end{cases}$$
			\end{lemma}
			
			\begin{proof}
				Let $\zeta$ be a generator of $\mu_{n}$, and let $\zeta_{0}=\omega_{\rho}(\zeta)$ which is a generator of $\mu_{s}$. Then the left-hand side becomes
				$$\sum_{k_{1}+\dots+k_{s}=m_{i}}\prod_{j=1}^{s}\binom{n/s}{k_{j}}\zeta_{0}^{jk_{j}},$$
				where $k_{j}$ in the sum are non-negative integers, meaning that for each $j$ there are $k_{j}$'s $\zeta_{il}$ with $1\leq l\leq m_{i}$ such that $\omega_{\rho}(\zeta_{il})=\zeta_{0}^{j}$.
				Considering the generating function
				$$\prod_{j=1}^{s}(\zeta_{0}^{j}X+1)^{n/s}=\big((-1)^{s+1}X^{s}+1\big)^{n/s}$$
				and comparing the coefficient of $X^{m_{i}}$, we finish the proof.
			\end{proof}
			Using this lemma, we further have
			\begin{align*}d_{\wt{\pi}}&=\frac{1}{d_{r}}\sum_{\substack{m_{1}+\dots+m_{l}=m\\ s\ \text{divides}\ m_{i}}}\bigg(\prod_{i=1}^{l}(-1)^{m_{i}-1}(-1)^{(s+1)m_{i}/s}\cdot\binom{n/s}{m_{i}/s}\bigg)\\
				&=\frac{1}{d_{r}}\sum_{m_{1}'+\dots+m_{l}'=m/s}\bigg(\prod_{i=1}^{l}(-1)^{m_{i}'-1}\cdot\binom{n/s}{m_{i}'}\bigg)\\
				&=\frac{1}{d_{r}}\cdot\binom{m/s+n/s-1}{m/s},
			\end{align*}
			where for the last step we compare the coefficient of $X^{m/s}$ in the generating function
			$$\sum_{i=0}^{+\infty}\big(\sum_{j=1}^{n/s}(-1)^{j-1}\binom{n/s}{j}X^{j}\big)^{i}=\frac{1}{(1-X)^{n/s}}=\sum_{i=0}^{+\infty}\binom{i+n/s-1}{i}X^{i},$$
			so we finish the proof.
			
		\end{proof}
		
		\begin{remark}\label{remcalirred}
			
			By Theorem \ref{thmwhitcal} and Proposition \ref{propwhitdimBZprod}, it is theoretically and algorithmically possible to calculate $d_{\wt{\pi}}$ for any $\wt{\pi}\in\mrirr_{\wt{\omega}}(\wt{G_{r}})$, since in the Grothendieck group $\wt{\pi}$ can be written as a linear combination (with coefficients being integers that can be determined) of representations of the form
			$$(L(\wt{\rho}_{1},[0,m_{1}-1])_{\wt{\omega}_{1}}\wt{\times}\dots\wt{\times} L(\wt{\rho}_{l},[0,m_{l}-1])_{\wt{\omega_{l}}})_{\wt{\omega}},$$
			where $\wt{\rho}_{i}\in\mrcusp_{\epsilon}(\wt{G_{r_{i}}})$ such that $\sum_{i=1}^{l}r_{i}m_{i}=r$ and $\wt{\omega}_{i}$ are genuine compatible characters of $Z(\wt{G_{r_{i}m_{i}}})$ for each $i$
			(\emph{cf.} \cite{kaplan2022classification}).
			
		\end{remark} 
		
		In particular we have the following corollary.
		
		\begin{corollary}\label{corwhitdimLZ}
			
			Let $\wt{\rho}\in\mrcusp_{\epsilon}(\wt{G_{r_{0}'}})$. Assume $\mrmc(\wt{\rho})=L(\rho,[0,s-1])$ for $r_{0}'=r_{0}s$ and $\rho\in\mrcusp(G_{r_{0}})$, such that $\omega_{\rho}\rest_{\mu_{n}}$ is of order $s$ (\emph{cf.} Proposition \ref{propmcsqrtclassification}). If $\mrgcd(n,p)=1$, then $$d_{L(\wt{\rho},[0,k-1])_{\wt{\omega}}}=\frac{1}{d_{r}}\binom{k+n/s-1}{k}\quad\text{and}\quad d_{Z(\wt{\rho},[0,k-1])_{\wt{\omega}}}=\frac{1}{d_{r}}\binom{n/s}{k}.$$
			
		\end{corollary}
		
		\begin{proof}
			
			The first statement is nothing but a reformulation of Theorem \ref{thmwhitcal}. For the second statement, recall that we have the following determinantal formula (\emph{cf.} \cite{kaplan2022classification}, \cite{lapid2014determinantal})
			\begin{align*}
				&Z(\wt{\rho},[0,k-1])_{\wt{\omega}}\\
				=&\sum_{k_{1}+\dots+k_{l}=k}(-1)^{k-l}(L(\wt{\rho},[t_{0},t_{1}-1])_{\wt{\omega}_{1}}\wt{\times}L(\wt{\rho},[t_{1},t_{2}-1])_{\wt{\omega}_{2}}\wt{\times}\dots\wt{\times}L(\wt{\rho},[t_{l-1},t_{l}-1])_{\wt{\omega}_{l}})_{\wt{\omega}},
			\end{align*}
			where we write $t_{0}=0$ and $t_{i}=k_{1}+\dots+k_{i}$ for each $i=1,\dots,l$, and the above two sums are in the Grothendieck group of finite length representations of $\wt{G_{r}}$. Taking the Whittaker dimension and using the first formula and Proposition \ref{propwhitdimBZprod}, we get
			$$d_{r}\cdot d_{Z(\wt{\rho},[0,k-1])_{\wt{\omega}}}=\sum_{k_{1}+\dots+k_{l}=k}\prod_{i=1}^{l}(-1)^{k_{i}-1}\binom{k_{i}+n/s-1}{k_{i}}=\binom{n/s}{k},$$
			where the last equation follows by considering the coefficient of $X^{k}$ in the generating function
			$$\sum_{i=0}^{+\infty}\big(\sum_{j=1}^{+\infty}(-1)^{j-1}\binom{j+n/s-1}{j}X^{j}\big)^{i}=\sum_{i=0}^{+\infty}(1-(1+X)^{-n/s})^{i}=(1+X)^{n/s}.$$
			
		\end{proof}
		
		In particular we have the following interesting special case.
		
		\begin{corollary}\label{corwhitdimZlarge}
			
			We keep the notation of Corollary \ref{corwhitdimLZ}. Then $$d_{Z(\wt{\rho},[0,n/s-1])_{\wt{\omega}}}=1\quad \text{and}\quad d_{Z(\wt{\rho},[0,k-1])_{\wt{\omega}}}=0\ \text{if}\ k>n/s.$$
			
		\end{corollary}
		
		It should be interesting to give an independent algebraic proof of the above corollary, like what Bernstein-Zelevinsky did in \cite{bernstein1977induced, zelevinsky1980induced} in the $n=1$ case. Their proof, which is mysterious enough, relies on the multiplicity one theorem of the Whittaker model of a generic representation. So giving an independent proof of the about corollary includes providing a better understanding of the old but prominent theory of Bernstein-Zelevinsky.
		
		\subsection{Evidence of Conjecture \ref{conjmain}}\label{subsectioneviconj}
		
		In this subsection we still let $G=G_{r}$. We verify Conjecture \ref{conjmain} for  two special cases.
		
		First we consider the case where $r_{0}=1$, then $\rho$ is indeed a character of $F^{\times}$. In this case by definition we have $$Z(\rho,[a,b])=(\rho\nu^{(b+a)/2})\circ\mrdet$$
		which is a character of $G_{r}$. Then for $x=\mrdiag(\zeta_{1},\cdots,\zeta_{r})$ with $\zeta_{i}\in\mu_{n}$ we have
		$$\Delta^{G}(x)\cdot c_{Z(\rho,[a,b])}(x)=\lim_{x'\rightarrow x}\Delta^{G}(x')\cdot(\rho\nu^{(b+a)/2})(\mrdet(x'))$$
		where in the limit $x'=\mrdiag(\zeta_{1}',\dots,\zeta_{r}')$ with $\zeta_{i}'\in\mu_{n}$ pairwise different. If there exist $i\neq j$ such that $\zeta_{i}=\zeta_{j}$, then $$\lim_{x'\rightarrow x}\Delta^{G}(x')=0\quad\text{and thus}\quad c_{Z(\rho,[a,b])}(x)=0;$$ 
		if $\zeta_{i}$ are pairwise different for $i=1,\dots,r$, we have $$c_{Z(\rho,[a,b])}(x)=(\rho\nu^{(b+a)/2})(\mrdet(x))=\prod_{i=1}^{r}\rho(\zeta_{i}).$$ So in this case the conjecture is verified. In particular we have the following important corollary as a special case of Corollary \ref{corwhitdimLZ}.
		
		\begin{corollary}
			
			Let $\wt{\chi}\in\mrirr_{\epsilon}(\wt{F^{\times}})$. Assume $\mrgcd(n,p)=1$, then
			$$d_{L(\wt{\chi},[0,k-1])_{\wt{\omega}}}=\frac{1}{d_{r}}\binom{k+n-1}{k}\quad\text{and}\quad d_{Z(\wt{\chi},[0,k-1])_{\wt{\omega}}}=\frac{1}{d_{r}}\binom{n}{k}.$$
			
		\end{corollary}
		
		\begin{remark}\label{remcalirredwhdim}
			
			In the case where $\wt{\chi}$ is an unramified representation,  $Z(\wt{\chi},[0,k-1])_{\wt{\omega}}$ is usually called an ``exceptional representation" or a ``theta representation", whose Whittaker dimension is calculated in \cite{kazhdan1984metaplectic}*{Theorem I.3.5} by considering functional equations related to intertwining operators. So we give an independent proof of their result and also generalize it to ramified case.
			
		\end{remark}
		
		Now we consider the other extreme for cuspidal representations, which is summed up as the following theorem.
		
		\begin{theorem}\label{thmcalcpi}
			
			Assume $\mrgcd(n,p)=1$. For $\pi\in\mrcusp(G)$ and $x=\mrdiag(\zeta_{1},...,\zeta_{r})$ with $\zeta_{i}\in\mu_{n}\subset\mbf_{q}^{\times}\subset F^{\times}$, we have
			$$c_{\pi}(x)=\begin{cases} \omega_{\pi}(\zeta)\quad&\text{if}\ \zeta_{1}=...=\zeta_{r}=\zeta;\\
				0\quad&\text{if}\ \zeta_{i}\neq\zeta_{j}\ \text{for some}\ i\neq j.  \end{cases}$$

		\end{theorem}
		
		\begin{proof}
			
			When $\zeta_{1}=...=\zeta_{r}=1$, by \cite{rodier1975modele} and \cite{shalika1974multiplicity} $c_{\pi}(I_{r})$ equals the Whittaker dimension of $\pi$ which is 1. For $\zeta_{1}=...=\zeta_{r}=\zeta\in\mu_{n}$ in general, by (\ref{eqcpilimtheta}) we have $c_{\pi}(\zeta I_{r})=\omega_{\pi}(\zeta)c_{\pi}(I_{r})=\omega_{\pi}(\zeta)$. To prove the rest, we first give an ad-hoc introduction of the simple type theory for our use, and we leave \cite{bushnell2019arithmetic}, or in general \cite{bushnell129admissible} for more details.
			
			A simple stratum $[\mfa,\beta]$ in $A=\mrm_{r}(F)$ consists of a hereditary order $\mfa$ in $A$ and an element $\beta\in G$, such that $E=F[\beta]$ is a field of degree $d$ over $F$, and $E^{\times}$ normalizes $\mfa$. Let $B\simeq \mrm_{r}(E)$ be the centralizer of $E$ in $A$, thus $r=md$. Let $\mfb=B\cap\mfa$ be a hereditary order in $B$, let $\mfp_{\mfa}$ be the Jacobson radical of $\mfa$ and let $\mfp_{\mfb}=\mfp_{\mfa}\cap\mfb$ be the Jacobson radical of $\mfb$. We further assume our simple stratum to be maximal, saying that $\mfb$ is a maximal order in $B$. In this case, we change $\mfa$ up to $G$-conjugacy, such that $$\mfa=\{(a_{ij})_{1\leq i,j\leq e}\mid a_{ij}\in\mrm_{mf}(\mfo_{F})\ \text{for}\ 1\leq i\leq j\leq e\ \ \text{and}\ a_{ij}\in\mrm_{mf}(\mfp_{F})\ \text{for}\ 1\leq j< i\leq e \}$$
			and $$\mfp_{\mfa}=\{(a_{ij})_{1\leq i,j\leq e}\mid a_{ij}\in\mrm_{mf}(\mfo_{F})\ \text{for}\ 1\leq i< j\leq e\ \ \text{and}\ a_{ij}\in\mrm_{mf}(\mfp_{F})\ \text{for}\ 1\leq j\leq i\leq e \},$$ where $d=ef$ with $e$ the ramification index and $f$ the residue degree of $E/F$, and $\mfo_{F}$ (resp. $\mfo_{E}$) denotes the ring of integers of $F$ (resp. $E$), and $\mfp_{F}$ (resp. $\mfp_{E}$) denotes the corresponding maximal ideal. We further have the $\mbf_{q}$-algebra embedding
			\begin{equation}\label{eqiotaba}
				\iota=(\iota_{1},...,\iota_{e}):\mrm_{m}(\mbf_{q^{f}})\simeq\mfb/\mfp_{\mfb}\hookrightarrow\mfa/\mfp_{\mfa}\simeq\underbrace{\mrm_{mf}(\mbf_{q})\times...\times\mrm_{mf}(\mbf_{q})}_{e\text{-copies}},
			\end{equation}
			where each $\iota_{i}$ is an $\mbf_{q}$-algebra embedding\footnote{Each $\iota_{i}$ is injective since $\mrm_{m}(\mbf_{q^{f}})$ is a simple $\mbf_{q}$-algebra.} from $\mrm_{m}(\mbf_{q^{f}})$ to the $i$-th $\mrm_{mf}(\mbf_{q})$ on the right-hand side, which induces the corresponding group embedding 
			\begin{equation}\label{eqiotab*a*}
				\iota=(\iota_{1},...,\iota_{e}):\mrgl_{m}(\mbf_{q^{f}})\simeq\mfb^{\times}/1+\mfp_{\mfb}\hookrightarrow\mfa^{\times}/1+\mfp_{\mfa}\simeq\underbrace{\mrgl_{mf}(\mbf_{q})\times...\times\mrgl_{mf}(\mbf_{q})}_{e\text{-copies}}.
			\end{equation}
			Associated with $[\mfa,\beta]$, we may construct an open compact pro-$p$-group $J^{1}(\mfa,\beta)\subset 1+\mfp_{\mfa}$ and an open compact group $J^{0}(\mfa,\beta)=\mfb^{\times} J^{1}(\mfa,\beta)\subset \mfa^{\times}$ , and we have $J^{0}(\mfa,\beta)/J^{1}(\mfa,\beta)\simeq\mfb^{\times}/1+\mfp_{\mfb}\simeq \mrgl_{m}(\mbf_{q^{f}})$. We will write $J^{0}$ and $J^{1}$ instead of $J^{0}(\mfa,\beta)$ and $J^{1}(\mfa,\beta)$ for short. 
			
			One main result of the simple type theory predicts the existence of a pair $(\bs{J},\Lambda)$ with $\bs{J}:=E^{\times}J^{0}$ and an irreducible representation $\Lambda$ of $\bs{J}$, such that $\pi$ is isomorphic to the compact induction $\mrind_{\bs{J}}^{G}(\Lambda)$. The restriction $\lambda=\Lambda\rest_{J^{0}}$ equals the tensor product $\kappa\otimes \rho$. Here $\kappa$ is a so-called \emph{$\beta$-extension} as an irreducible representation of $J^{0}$ of dimension a power of $p$, whose construction and properties shall be irrelevant. And $\rho$ is the inflation of a cuspidal representation $\overline{\rho}$ of $\mrgl_{m}(\mbf_{q^{f}})\simeq\mfb^{\times}/1+\mfp_{\mfb}$. 
			
			We come back to the original proof. Let $x'=\mrdiag(\zeta_{1}',...,\zeta_{r}')$ be a semi-simple regular element of $G$, such that $\zeta_{i}'\in\zeta_{i}(1+\mfp_{F})$ for each $i$. For a finite dimensional irreducible representation $\Lambda$ of an $\ell$-group $\bs{J}$, we denote by $\mrtr(\Lambda)$ the trace of $\Lambda$. Using \cite{bushnell1996local}*{Theorem A.14}, for an open compact subgroup $K_{1}$ of $G$, we have 
			\begin{equation}\label{eqthetapiTrLambda}
				\theta_{\pi}(x')=\sum_{h\in K_{1}\backslash G/\bs{J}}\sum_{g\in K_{1}h\bs{J}/\bs{J}, g^{-1}x'g\in\bs{J}}\mrtr(\Lambda)(g^{-1}x'g).
			\end{equation}
			
			\begin{lemma}\label{lemmatrLambda=0}
				
				If $\zeta_{i}\neq \zeta_{j}$ for some $1\leq i<j\leq n$, then for any $g\in G$ such that $g^{-1}x'g\in\bs{J}$ we have $\mrtr(\Lambda)(g^{-1}x'g)=0$. 
				
			\end{lemma}
			
			It is clear that the rest of Theorem \ref{thmcalcpi} follows from Lemma \ref{lemmatrLambda=0}, (\ref{eqcpilimtheta}) and (\ref{eqthetapiTrLambda}), so we focus on the proof of Lemma \ref{lemmatrLambda=0}. Since $\mrdet(g^{-1}x'g)=\prod_{i=1}^{r}\zeta_{i}'\in\mfo_{F}^{\times}$, we have $g^{-1}x'g\in J^{0}$. We write $g^{-1}x'g=yk$ with $y\in\mfb^{\times}, k\in J^{1}$, and $\overline{y}$ the image of $y$ in $\mrgl_{m}(\mbf_{q^{f}})\simeq\mfb^{\times}/1+\mfp_{\mfb}$.
			
			\begin{lemma}\label{lemmacharpoly}
				
				Let $P_{\overline{y}}(X)$ be the characteristic polynomial of $\overline{y}$ in $\mrm_{m}(\mbf_{q^{f}})$, then $$[\mrn_{\mbf_{q^{f}}/\mbf_{q}}(P_{\overline{y}}(X))]^{e}=\prod_{i=1}^{r}(X-\zeta_{i}),$$ 
				where $\mrn_{\mbf_{q^{f}}/\mbf_{q}}:=\prod_{\sigma\in\mathrm{Gal}(\mbf_{q^{f}}/\mbf_{q})}\sigma$ denotes the norm map.
				
			\end{lemma}
			
			\begin{proof}
				
				By definition, $P'(X)=\prod_{i=1}^{r}(X-\zeta_{i}')\in\mfo_{F}[X]$ is the characteristic polynomial of $g^{-1}x'g=yk$, thus $P(X)=\prod_{i=1}^{r}(X-\zeta_{i})\in \mbf_{q}[X]$ is the characteristic polynomial of $yk$ in $\mrgl_{r}(\mbf_{q})\simeq\mrgl_{r}(\mfo_{F})/1+\mrm_{r}(\mfp_{F})$. Since $yk\in\mfa^{\times}$, by direct calculation $P(X)$ depends only on the image of $yk$ in $\mrgl_{mf}(\mbf_{q})^{\oplus e}\simeq\mfa^{\times}/1+\mfp_{\mfa}$. More precisely by (\ref{eqiotab*a*}), $\iota(\overline{y})=\mrdiag(\iota_{1}(\overline{y}),...,\iota_{e}(\overline{y}))$ is exactly the image of $yk$ in $\mrgl_{mf}(\mbf_{q})^{\oplus e}$.   
				Moreover, $P(X)=\prod_{i=1}^{e}P_{\iota_{i}(\overline{y})}(X)$ with $P_{\iota_{i}(\overline{y})}(X)$ denoting the characteristic polynomial of $\iota_{i}(\overline{y})$ in $\mrm_{mf}(\mbf_{q})$. Finally we have $\mrn_{\mbf_{q^{f}}/\mbf_{q}}(P_{\overline{y}}(X))=P_{\iota_{i}(\overline{y})}(X)$ for each $i$, thus
				$[\mrn_{\mbf_{q^{f}}/\mbf_{q}}(P_{\overline{y}}(X))]^{e}=\prod_{i=1}^{e}P_{\iota_{i}(\overline{y})}(X)=P(X)=\prod_{i=1}^{r}(X-\zeta_{i}).$
				
			\end{proof}
			
			Now we finish the proof of Lemma \ref{lemmatrLambda=0}. By definition we have $\mrtr(\Lambda)(g^{-1}x'g)=\mrtr(\kappa)(g^{-1}x'g)\cdot\mrtr(\rho)(g^{-1}x'g)$ and $\mrtr(\rho)(g^{-1}x'g)=\mrtr(\rho)(yk)=\mrtr(\overline{\rho})(\overline{y})$. Using \cite{green1955characters}*{Theorem 12},  $\mrtr(\overline{\rho})(\overline{y})=0$ unless $P_{\overline{y}}(X)$ is a power of an irreducible polynomial in $\mbf_{q}[X]$, which happens only if $\zeta_{1}=...=\zeta_{r}$
			using Lemma \ref{lemmacharpoly}. Thus we must have $\mrtr(\overline{\rho})(\overline{y})=0$ and $\mrtr(\Lambda)(g^{-1}x'g)=0$ once there exist $i,j$ such that $\zeta_{i}\neq\zeta_{j}$, finishing the proof of Lemma \ref{lemmatrLambda=0} and Theorem \ref{thmcalcpi}.
			
		\end{proof}
		
		As a special case of Theorem \ref{thmwhitcal}, we have
		
		\begin{corollary}
			
			Assume $\mrgcd(n,p)=1$. Let $\wt{\pi}\in\mrcusp_{\epsilon}(\wt{G})$ such that $\pi=\mrmc(\wt{\pi})\in\mrcusp(G)$, then
			$$d_{\wt{\pi}}=n/d_{r}.$$
			
		\end{corollary}
		
		\begin{remark}
			
			We notice that $\wt{\pi}$ is cuspidal if $\pi$ is so, but the converse is not true in general (\emph{cf.} Proposition \ref{propmcsqrtclassification}). So this corollary calculates the Whittaker dimension of a certain class of cuspidal representations of $\wt{G}$. One may also compare our result with the result of Blondel (\cite{blondel1992uniqueness}*{Theorem 3}) for depth 0 cuspidal representations, and find out that they match well if the number $k$ in \emph{loc. cit.} is 1, which is indeed expected when $\pi$ is cuspidal (although we do not give a proof here). This ``coincidence" indeed motivates us to study an ``explicit" metaplectic correspondence as what Bushnell-Henniart \cite{bushnell2011essentially} proposed for the Jacquet-Langlands correspondence.
			
		\end{remark}
		
		\section{Conjecture 14 of \cite{kaplan2022rankin} for a Kazhdan-Patterson covering group}
		
		In this section, we explain how our results could be used to study an analogue of \cite{kaplan2022rankin}*{Conjecture 14} for a Kazhdan-Patterson covering group. Our discussion here is rather heuristic.
		
		Still, all the covering groups we consider are Kazhdan-Patterson for fixed $n,c$. Fix two positive integers $r$ and $c_{0}$. Let $\wt{\pi}\in\mrtemp_{\epsilon}(\wt{G_{r}})$, which is necessarily generic by Proposition \ref{proptempclasify}, Proposition \ref{propwhitdimBZprod} and Corollary \ref{corwhitdimLZ}. For a compatible genuine character $\wt{\omega}$ of $Z(\wt{G_{rnc_{0}}})$, we consider the parabolic induction
		\begin{equation}\label{eqparindrhoc}
			(\wt{\pi}\nu^{-(c_{0}n-1)/2n}\wt{\times}\wt{\pi}\nu^{-(c_{0}n-3)/2n}\wt{\times}\dots\wt{\times}\wt{\pi}\nu^{(c_{0}n-3)/2n}\wt{\times}\wt{\pi}\nu^{(c_{0}n-1)/2n})_{\wt{\omega}}.
		\end{equation}
		By the Langlands classification  \cite{ban2013langlands}*{Theorem 4.1}, it has a unique irreducible subrepresentation which we denote by $\rho_{c_{0}}(\wt{\pi})_{\wt{\omega}}$. We state  \cite{kaplan2022rankin}*{Conjecture 14} in our settings.
		
		\begin{conjecture}
			
			When $c_{0}=1$, the Whittaker space of $\rho_{1}(\wt{\pi})_{\wt{\omega}}$ is one-dimensional. In general, $\rho_{c_{0}}(\wt{\pi})_{\wt{\omega}}$ is a $(nr,c_{0})$-representation (\emph{cf.} \cite{kaplan2022rankin}*{Definition 11}), or, a fortiori, the highest derivative of $\rho_{c_{0}}(\wt{\pi})_{\wt{\omega}}$ is of order $nr$ and equals $\nu^{-1/2}\cdot\rho_{c_{0}-1}(\wt{\pi})_{\wt{\omega}'}$ for a certain compatible genuine character $\wt{\omega}'$ of $Z(\wt{G_{rn(c_{0}-1)}})$.
			
		\end{conjecture}
		
		\begin{remark}
			
			The exponent $-1/2$ in $\nu^{-1/2}\cdot\rho_{c_{0}-1}(\wt{\pi})_{\wt{\omega}'}$  differs from that in \cite{kaplan2022rankin}*{Conjecture 14} (which is $(n-1)/2$). It is because in \emph{loc. cit.} the author considered non-normalized derivative functors, whereas we consider normalized derivative functors.
			
		\end{remark}
		
		We expect that using Corollary \ref{corwhitdimLZ}, the Bernstein-Zelevinsky theory of derivatives in \cite{bernstein1976representations}*{\S 3-4} and the theory of generalized Speh representations studied by Tadi\'c \cite{tadic1986classification} and others, this conjecture can be verified as in the linear case \cite{cai2021generalized}*{Theorem 4}. However, we need to generalize the above results to a Kazhdan-Patterson covering group, which, although expected to be routine, has not been fully done yet. So instead we satisfy ourselves with a special case, where $\wt{\pi}$ is a cuspidal representation.
		
		\begin{proposition}
			
			When $\wt{\pi}$ is a cuspidal representation, the above conjecture is true.
			
		\end{proposition}
		
		\begin{proof}
			
			Let $s(\wt{\pi})$ be as in Proposition \ref{propvaluesrho}, let $s=s(\wt{\pi})n$ be a positive integer that divides $n$ and let $m=n/s$. Let 
			$$\wt{\pi}_{0}=Z(\wt{\pi},[-(mc_{0}-1)/2,(mc_{0}-1)/2])_{\wt{\omega}_{0}}$$
			be the unique irreducible subrepresentation of
			$$(\nu^{-(mc_{0}-1)s/2n} \wt{\pi}\wt{\times}\nu^{-(mc_{0}-3)s/2n} \wt{\pi}\wt{\times}\dots\wt{\times}\nu^{(mc_{0}-1)s/2n}\wt{\pi})_{\wt{\omega}_{0}}$$
			of $\wt{G_{rmc_{0}}}$, where $\wt{\omega}_{0}$ is a certain compatible genuine character of $Z(\wt{G_{rmc_{0}}})$. Using the Zelevinsky classification \cite{kaplan2022classification}, the representation \eqref{eqparindrhoc} has an irreducible subrepresentation
			\begin{equation*}
				(\nu^{-(s-1)/2n} \wt{\pi}_{0}\wt{\times}\nu^{-(s-3)/2n} \wt{\pi}_{0}\wt{\times}\dots\wt{\times}\nu^{(s-1)/2n}\wt{\pi}_{0})_{\wt{\omega}},
			\end{equation*}
			which is exactly $\rho_{c_{0}}(\wt{\pi})_{\wt{\omega}}$.
			Using \cite{kaplan2022rankin}*{Proposition 12} whose argument works for a Kazhdan-Patterson covering group as well, we need to prove that
			\begin{itemize}
				\item When $c_{0}=1$, the Whittaker space of $Z(\wt{\pi},[-(m-1)/2,(m-1)/2])_{\wt{\omega}_{0}}$ is one-dimensional;
				\item $\wt{\pi}_{0}=Z(\wt{\pi},[-(mc_{0}-1)/2,(mc_{0}-1)/2])_{\wt{\omega}_{0}}$ is an $(mr,c_{0})$-representation. Or, a fortiori, the highest derivative of $Z(\wt{\pi},[-(mc_{0}-1)/2,(mc_{0}-1)/2])_{\wt{\omega}_{0}}$ is of order $mr$ and equals  
				\begin{align*}
					&Z(\wt{\pi},[-(mc_{0}-1)/2,(mc_{0}-1-2m)/2])_{\wt{\omega}_{0}'}\\
					=&\nu^{-1/2}\cdot Z(\wt{\pi},[-(mc_{0}-1-m)/2,(mc_{0}-1-m)/2])_{\wt{\omega}_{0}''},
				\end{align*}
				where $\wt{\omega}_{0}'$ and $\wt{\omega}_{0}''$ are certain compatible genuine characters of $Z(\wt{G_{r(m-1)c_{0}}})$.
			\end{itemize}
			The first claim follows from Corollary \ref{corwhitdimZlarge}. 
			The second claim follows from Corollary \ref{corwhitdimZlarge}, the formula of Jacquet module of $\wt{\pi}_{0}$ in Lemma \ref{lemmaLabwtrho}.(2), and the fact that the $i$-th derivative functor equals the composition of the Jacquet module functor related to the composition $(rmc_{0}-i,i)$ with the $i$-th second partial derivative functor. We omit the detail but leave \cite{zelevinsky1980induced}*{\S 3} for a similar argument.
			
		\end{proof}
		

	\end{document}